\documentclass[12pt,twoside]{amsart}

\usepackage{amsfonts,amsmath,amssymb}
\usepackage{blkarray}
\usepackage{pinlabel}
\usepackage{enumerate}
\usepackage{mathtools}
\usepackage{xcolor,colortbl}
\usepackage{url}
\usepackage[hidelinks]{hyperref}
\usepackage{cleveref}
\usepackage{subcaption}

\usepackage{tikz}
\usetikzlibrary{cd,matrix,knots}

\def\filled{\path[fill=black, fill opacity=0.15, even odd rule]}
\def\stroked{\draw[line cap=round, line join=miter, ultra thick]}
\def\clearstroked{\draw[line cap=butt, line width=3,white]}

\newcommand{\edge}{%
  \begin{tikzpicture}
    \draw[thick] (0,0) -- (0.5,0);
    \filldraw [black] (0,0) circle (2pt);
    \filldraw [black] (0.5,0) circle (2pt);
  \end{tikzpicture}
    }
    
\newcommand{\singleton}{%
  \begin{tikzpicture}
    \filldraw [black] (0,0) circle (2pt);
  \end{tikzpicture}
    }

\newcommand{\round}{%
  \begin{tikzpicture}[baseline=(base)]
    \draw [ultra thick,black,fill=black,fill opacity=0.15] (0,0) circle (4pt);
    \node (base) at (0,-.6ex) {};
  \end{tikzpicture}
    }

\newcommand{\onecrossing}{%
  \begin{tikzpicture}[baseline=(base)]
    \coordinate (ll) at (0,0);
    \coordinate (lt) at (0.12,0.14);
    \coordinate (lb) at (0.12,-0.14);
    \coordinate (rt) at (0.48,0.14);
    \coordinate (rb) at (0.48,-0.14);
    \coordinate (rr) at (0.6,0);
    \stroked (ll) .. controls +(0,-0.1) and +(-0.05,0) .. (lb) .. controls +(0.2,0) and +(-0.2,0) .. (rt) .. controls +(0.05,0) and +(0,0.1) .. (rr);
    \clearstroked (ll) .. controls +(0,0.1) and +(-0.05,0) .. (lt) .. controls +(0.2,0) and +(-0.2,0) .. (rb) .. controls +(0.05,0) and +(0,-0.1) .. (rr);
    \stroked (ll) .. controls +(0,0.1) and +(-0.05,0) .. (lt) .. controls +(0.2,0) and +(-0.2,0) .. (rb) .. controls +(0.05,0) and +(0,-0.1) .. (rr);
    \filled (ll) .. controls +(0,-0.1) and +(-0.05,0) .. (lb) .. controls +(0.2,0) and +(-0.2,0) .. (rt) .. controls +(0.05,0) and +(0,0.1) .. (rr) .. controls +(0,-0.1) and +(0.05,0) .. (rb) .. controls +(-0.2,0) and +(0.2,0) .. (lt) .. controls +(-0.05,0) and +(0,0.1) .. (ll);
    \node (base) at (0,-.6ex) {};
  \end{tikzpicture}
    }

\newtheorem{thm}{Theorem}[section]
\newtheorem{conj}[thm]{Conjecture}
\newtheorem{cor}[thm]{Corollary}
\newtheorem{lem}[thm]{Lemma}
\newtheorem{prop}[thm]{Proposition}
\newtheorem{ques}[thm]{Question}

\newtheorem{example}[thm]{Example}

\newtheorem{maintheorem}{Theorem}
\newtheorem{mainconjecture}[maintheorem]{Conjecture}
\newtheorem{maincor}[maintheorem]{Corollary}

\def\cG{{\mathcal G}}

\def\C{{\mathbb C}}
\def\P{{\mathbb P}}
\def\cD{{\mathcal D}}

\def\bQ{{\mathbb Q}}
\def\bR{{\mathbb R}}
\def\cT{{\mathcal T}}
\def\cX{{\mathcal X}}
\def\bZ{{\mathbb Z}}

\def\qq{{\mathbb Q}}
\def\rr{{\mathbb R}}
\def\zz{{\mathbb Z}}

\def\Char{{\mathrm{Char}}}
\def\Cut{{\mathrm{Cut}}}
\def\Flow{{\mathrm{Flow}}}
\def\Hom{{\mathrm{Hom}}}

\def\spc{{\mathrm{spin^c}}}
\def\Spc{{\mathrm{Spin^c}}}
\def\Short{{\mathrm{Short}}}

\def\cut{{\textup{Cut}}}

\def\del{{\partial}}
\def\ud{{\underline{\det}}}

\def\im{{\text{im}}}

\def\rk{{\mathrm{rk}}}

\def\supp{{\textup{supp}}}

\newcommand{\into}{\hookrightarrow}

\newcommand{\longto}{\longrightarrow}

\newcommand{\hajos}{Haj\'{o}s}

\newcommand{\ozsvath}{Ozsv\'{a}th}
\newcommand{\szabo}{Szab\'{o}}
\newcommand{\spin}{\ifmmode{\mathrm{Spin}}\else{$\mathrm{spin}$\ }\fi}
\newcommand{\spinc}{\ifmmode{\mathrm{Spin}^c}\else{$\mathrm{spin}^c$\ }\fi}
\newcommand{\spinct}{\mathfrak t}
\newcommand{\spincs}{\mathfrak s}

\newcommand{\tors}{{\mathrm{Tors}}}
\newcommand{\flow}{\mathrm{Flow}}

\DeclareMathOperator\Id{Id}

\begin{document}

\title{Alternating links, rational balls, and cube tilings}
\author[Joshua Evan Greene]{Joshua Evan Greene}
\address{Department of Mathematics \newline\indent 
Boston College \newline\indent
USA }
\email{joshua.greene@bc.edu}
\author[Brendan Owens]{Brendan Owens}
\address{School of Mathematics and Statistics \newline\indent 
University of Glasgow \newline\indent 
Glasgow, G12 8QQ, United Kingdom}
\email{brendan.owens@glasgow.ac.uk}
\date{\today}
\thanks{JEG was supported on a Simons Fellowship and NSF grant DMS-2005619}

\begin{abstract}
When does the double cover of the three-sphere branched along an alternating link bound a rational homology ball?
Heegaard Floer homology generates a necessary condition for it to bound: the link's chessboard lattice must be cubiquitous, implying that its normalized determinant is less than or equal to one.
We conjecture that the converse holds and prove it when the normalized determinant equals one.
The proof involves flows on planar graphs and the Haj\'os-Minkowski theorem that a lattice tiling of Euclidean space by cubes contains a pair of cubes which touch along an entire facet.
We extend our main results to the study of ribbon cobordism and ribbon concordance.
\end{abstract}

\maketitle

\pagestyle{myheadings}
\markboth{JOSHUA EVAN GREENE AND BRENDAN OWENS}{ALTERNATING LINKS, RATIONAL BALLS, AND CUBE TILINGS}


\section{Introduction}
\label{sec:intro}

Casson raised the open-ended problem to determine which rational homology 3-spheres bound smooth rational homology 4-balls \cite[Problem 4.5]{kirby}.
Lisca solved the problem for lens spaces in an influential paper \cite{lisca}.
His method is to take a putative rational homology ball filling a lens space and to glue it along its boundary to an appropriate linear plumbing of disk bundles over spheres.
The result is a smooth, closed, oriented, definite 4-manifold, so by Donaldson's theorem, its intersection pairing is isometric to the lattice of integer points in Euclidean space of some dimension.
As a consequence, the intersection pairing of the linear plumbing embeds with full rank into the Euclidean lattice.
Lisca showed the converse: whenever such lattice embeddings exist for the linear plumbings which fill both a lens space and its orientation reversal, the lens space does bound a rational homology ball.

Lisca's work established a surprising theme: a necessary lattice embedding obstruction from Donaldson's theorem, plus a twist, is sufficient; moreover, the examples which pass it fall to a recursive construction.
Both the solution of the lens space realization problem and McCoy's theorem on alternating knots with unknotting number one can be viewed in this light \cite{lens,mccoy}.

The motivation for this work is to extend Lisca's result to double covers of the three-sphere branched along an alternating link.
The pertinent lattice embedding obstruction here is {\em cubiquity}, which appeared earlier, if somewhat obscurely, in \cite{GJ}.
In line with the theme, we conjecture that the obstruction is sufficient (\Cref{conj:main}), and we prove it under a natural numerical condition involving the normalized determinant (\Cref{thm:links2}).
We take a step further by establishing a generalization of cubiquity for ribbon rational homology cobordisms between branched double covers of alternating links (\Cref{thm:ribboncubiquity}).
Once more, we conjecture that the obstruction is sufficient (\Cref{conj:main2}), and we prove it under a condition on the normalized determinant (\Cref{thm:generalization2}).

\subsection{The main result.}
We work in the smooth category.
A link $L \subset S^3$ with $\det(L) \ne 0$ is {\em $\chi$-slice} if it bounds a properly embedded surface $F \subset D^4$ with $\chi(F) = 1$ and no closed components.
It is {\em $\chi$-ribbon} if it bounds a $\chi$-slice surface $F$ for which the radial distance function on $D^4$ restricts to a Morse function on $F$ without index-2 critical points; equivalently, there exists a cobordism in $S^3 \times [0,1]$ from $\bigcirc \times \{0\}$ to $L \times \{1\}$ built from 0- and 1-handles.
Let $\Sigma(L)$ denote the double cover of $S^3$ branched along $L$, and let $\Sigma(F)$ denote the double cover of $D^4$ branched along $F$.
If $L$ is $\chi$-slice, then $\Sigma(F)$ is a rational homology ball with boundary $\Sigma(L)$ \cite[Proposition 2.6]{DO}.
Hence for $\det(L) \ne 0$, we have the implications
\begin{center}
$L$ $\chi$-ribbon \quad $\implies$ \quad $L$ $\chi$-slice \quad $\implies$ \quad $\Sigma(L)$ rationally nullbordant.
\end{center}
Moreover, the set of links having any one of these three properties is closed under arbitrary connected sum.

An alternating link has determinant zero if and only if it is split.
Let $L$ be a non-split alternating link and $D$ an alternating diagram of $L$.
We give the regions of $D$ a chessboard coloring according to the convention that near each crossing the coloring appears as in 
\onecrossing.

The black Tait graph of $D$ is a planar graph $G(D)$ which has one vertex in each black region and one edge through each crossing of $D$ joining the vertices in the incident black regions.
Let $\Lambda(D)$ denote the lattice of integer-valued flows on $G(D)$.
The resolution of the Tait flyping conjecture \cite{MT}, or alternatively the main result of \cite{mut}, leads to the following result: if $D$ is reduced, and if $D'$ is another alternating diagram of $L$, then 
$\Lambda(D')$ is isometric to a stabilization $\Lambda(D) \oplus \bZ^n$ for some $n \ge 0$, where $\bZ^n$ denotes the lattice of integer points in Euclidean space $\bR^n$, and $n=0$ if and only if $D'$ is reduced.
Thus, we can unambiguously define $\Lambda(L)$ to be the isometry type of $\Lambda(D)$, where $D$ denotes a reduced alternating diagram of $L$.

A lattice is {\em cubiquitous} if it admits an embedding into some $\bZ^n$ in such a way that its image $\Lambda$ contains a point in each unit cube with integer vertices:
\[
\textup{cubiquity:} \quad \Lambda \cap (x + \{0,1\}^n) \ne \emptyset, \quad \forall x \in \bZ^n.
\]
The following result follows from \cite{GJ}:
\begin{maintheorem}
\label{thm:GJ1}
If $\Sigma(L)$ bounds a rational homology ball, then $\Lambda(L)$ is cubiquitous.
\end{maintheorem}
\noindent
We rederive \Cref{thm:GJ1} in \Cref{sec:cubiquity}.
We conjecture here that the converse holds:

\begin{mainconjecture}
\label{conj:main}
If $\Lambda(L)$ is cubiquitous, then $\Sigma(L)$ bounds a rational homology ball.
\end{mainconjecture}
\noindent
Moreover, we expect that there exists an algorithm to construct a rational ball filling $\Sigma(L)$ if $\Lambda(L)$ is cubiquitous.

Suppose that $\Lambda \subset \bZ^n$ is cubiquitous.
By \Cref{prop:cubiq}, the cube $\{0,1\}^n$ contains a point from each coset of $\Lambda$, leading to the inequality $[\bZ^n : \Lambda] \le 2^n$.
The first value is the square root of the lattice determinant $[\Lambda^* : \Lambda]$, and the second is 2 raised to the rank $\rk(\Lambda)$.
The inequality motivates a definition: for a lattice $\Lambda$, the {\em normalized determinant} is the value
\[
\ud(\Lambda) := \det(\Lambda) / 4^{\rk(\Lambda)}.
\]
Thus, $\ud(\Lambda) \le 1$ if $\Lambda$ is cubiquitous.
We define in turn the normalized determinants of a chessboard-colored alternating link diagram $D$ and an alternating link $L$ by setting
\[
\ud(D) := \ud(\Lambda(D)), \quad \ud(L) := \ud(\Lambda(L)).
\]
Thus, \Cref{thm:GJ1} implies that $\ud(L) \le 1$ if $\Sigma(L)$ bounds a rational homology ball.

Our main result, \Cref{thm:links2}, establishes \Cref{conj:main} under the assumption that $\ud(L) = 1$.
Observe that if $\Lambda$ is cubiquitous, then $\ud(\Lambda) = 1$ if and only if each unit cube in $\bZ^n$ contains a unique point of $\Lambda$.
Equivalently, centering a cube of side length 2 at each point of $\Lambda$, we obtain a lattice tiling of $\bR^n$ by cubes.
We call $\Lambda$ a {\em 2-cube tiling lattice} in this case.
See \Cref{fig:tiling}.
Lattice tilings by cubes are highly constrained.
Minkoswki conjectured in 1896 that in any lattice tiling by cubes, every cube must intersect some other cube in an entire codimension-1 face of each \cite{Mink}.
\hajos\ proved Minkowski's conjecture in 1942 \cite{Hajos}.
(The references \cite{SS,Zong} describe an interesting history and related results.)
\hajos's theorem implies that every 2-cube tiling lattice in $\bZ^n$ has, up to reordering the orthonormal basis of $\bZ^n$, a basis consisting of the columns of a lower triangular matrix $H$ with 2's on its diagonal and 0's and 1's below (\Cref{thm:tiling}).
Using \Cref{thm:tiling}, we determine when the lattice of flows on a planar graph is a 2-cube tiling lattice (\Cref{thm:2cubeflow}).
We show that such graphs admit a simple recursive construction.
Moreover, for each such graph, the corresponding alternating link $L$ is $\chi$-ribbon, so $\Sigma(L)$ bounds a rational ball.

\begin{figure}
\includegraphics[width=6in]{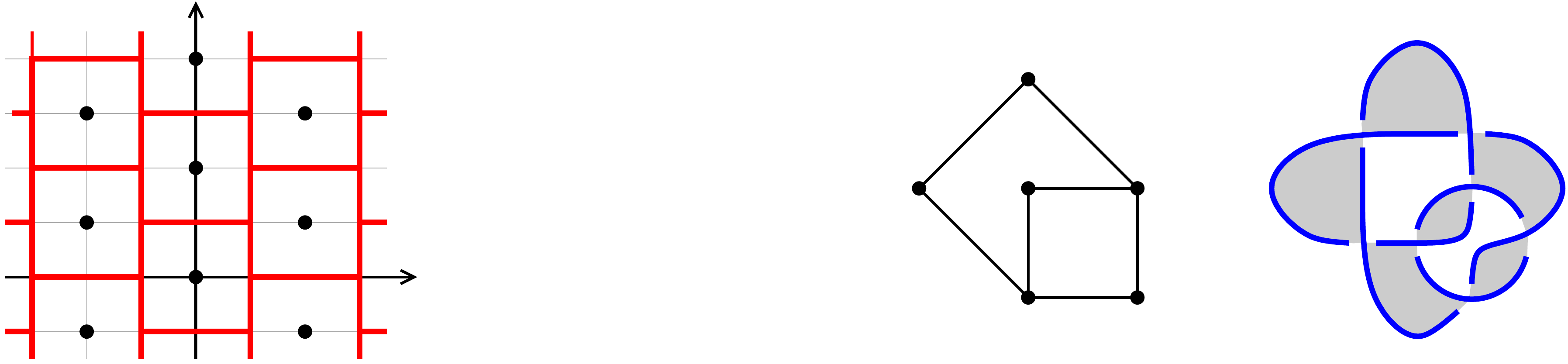}
\put(-315,20){$e_1$}
\put(-380,105){$e_2$}
\put(-346,43){$h_1$}
\put(-377,59){$h_2$}
\put(-295,45){
$H = \begin{blockarray}{ccc}
h_1 & h_2 &  \\
\begin{block}{[cc]c}
  2 & 0 & e_1 \\
  1 & 2 & e_2 \\
\end{block}
\end{blockarray}$
} 
\put(-160,55){$C_1$}
\put(-140,30){$C_2$}
\put(-180,80){$G$}
\put(-80,80){$D$}
\caption{
\label{fig:tiling}
(a) A 2-cube tiling lattice $\Lambda \subset \bR^2$. \\ (b) A \hajos\ matrix $H$ with $\Lambda_H = \Lambda$. \\ (c) A graph $G$ with $\flow(G) = \Lambda_H$ and cycles $C_i$ such that $[C_i] = h_i$. \\ (d) A chessboard-colored alternating diagram $D$ with black graph $G$ and $\Lambda(D) = \flow(G) = \Lambda_H = \Lambda$.}
\end{figure}
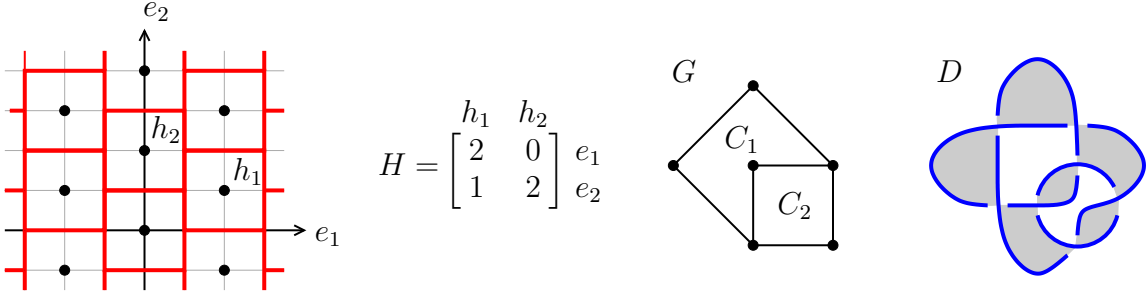

We now describe the construction and the main result in the language of link diagrams.
Define a composition of nontrivial chessboard-colored alternating diagrams $D_1$ and $D_2$, as shown in \Cref{fig:D}.
In words,
\begin{itemize}
\item
remove a crossing ball from each of $D_1$ and $D_2$;
\item
form a tangle sum of the resulting pair of tangles $\cT_1$ and $\cT_2$;
\item
introduce an unknotted component into the diagram along the circle where the tangle sum takes place; and
\item
add crossings and color consistently with $D_1$ and $D_2$ to obtain a chessboard-colored alternating diagram $D$.
\end{itemize}
Note that following the choice of crossing balls in the first step, there are two ways to make a tangle sum in the second step to ensure consistency in the final step.
Recursively define a set $\cD$ of chessboard-colored alternating diagrams by declaring that \round, \onecrossing $\in \cD$ and that $D \in \cD$ whenever $D$ is a composition of $D_1, D_2 \in \cD$.

\begin{maintheorem}
\label{thm:links2}
Let $D$ be a chessboard-colored alternating diagram of a link $L$ with $\ud(D) = \ud(L)=1$.
The following are equivalent:
\begin{enumerate}
\item
$L$ is $\chi$-ribbon;
\item
$L$ is $\chi$-slice;
\item
$\Sigma(L)$ is rationally nullbordant;
\item
the lattice $\Lambda(D)$ is a 2-cube tiling lattice; and
\item
$D$ is a connected sum of diagrams in $\cD$.
\end{enumerate}
Moreover, $D$ and $L$ are prime if and only if $D \in \cD$, and $D$ is reduced if and only if it is a connected sum of diagrams in $\cD \setminus \{\text{\onecrossing}\!\!\}$.
\end{maintheorem}

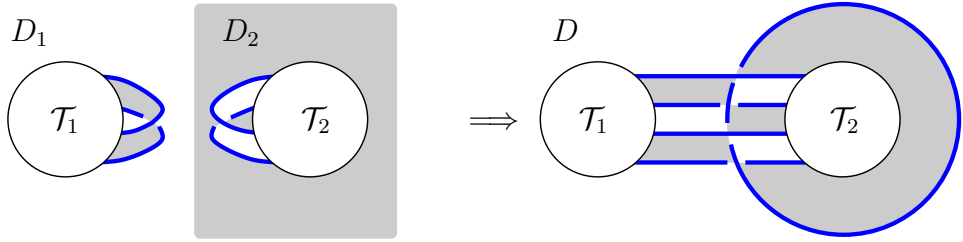
\begin{figure}
\includegraphics[width=5in]{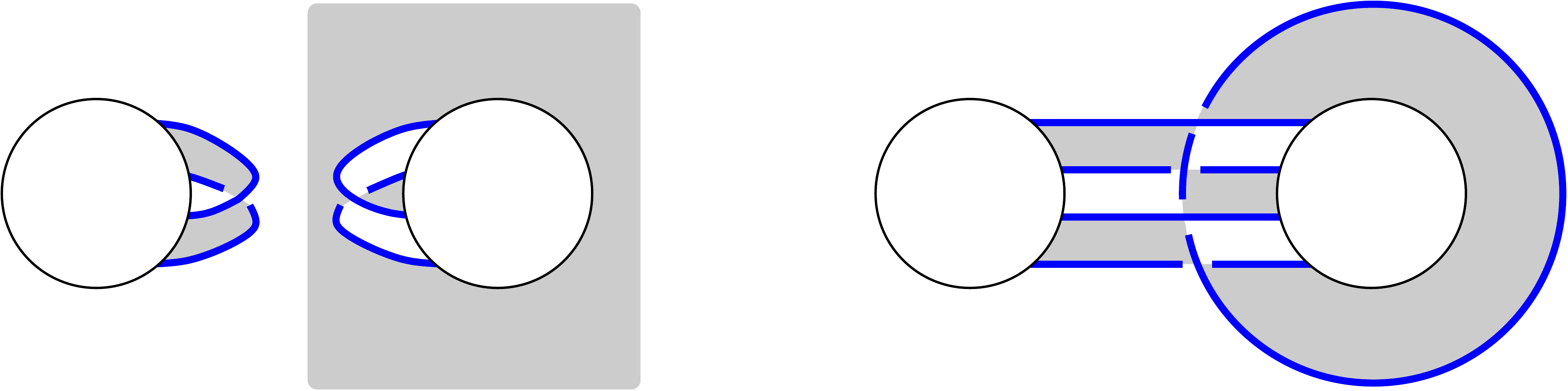}
\put(-360,75){$D_1$}
\put(-345,42){$\mathcal{T}_1$}
\put(-280,75){$D_2$}
\put(-250,42){$\mathcal{T}_2$}
\put(-190,42){$\implies$}
\put(-155,75){$D$}
\put(-145,42){$\mathcal{T}_1$}
\put(-50,42){$\mathcal{T}_2$}
\caption{
\label{fig:D}
The recursive construction of $\cD$.}
\end{figure}

\begin{proof}
The implications (1)$\implies$(2)$\implies$(3) were described above.
The implication (3)$\implies$(4) was as well, and we supply the details in \Cref{cor:altcubiquity} and in the paragraph following \Cref{prop:cubiq}.
The implication (4)$\implies$(5) is contained in \Cref{cor:2cubediagram}, which is the culmination of Sections \ref{sec:cubes}, \ref{sec:graphs}, and \ref{sec:cubegraphs}.
Lastly, for the implication (5)$\implies$(1), observe that if $D_1$ and $D_2$ compose to $D$, then there exists a  cobordism involving a single 0-handle and a single 1-handle from a connected sum $L_1 \# L_2$ to $L$, as displayed in \Cref{fig:bandmove} (diagram $D_*$ represents link $L_*$).
Since the unknot $L_0$ is trivially $\chi$-ribbon and connected sum preserves $\chi$-ribbonness, the recursive construction of $\cD$ gives (5)$\implies$(1).
\end{proof}

\begin{figure}
\includegraphics[width=6in]{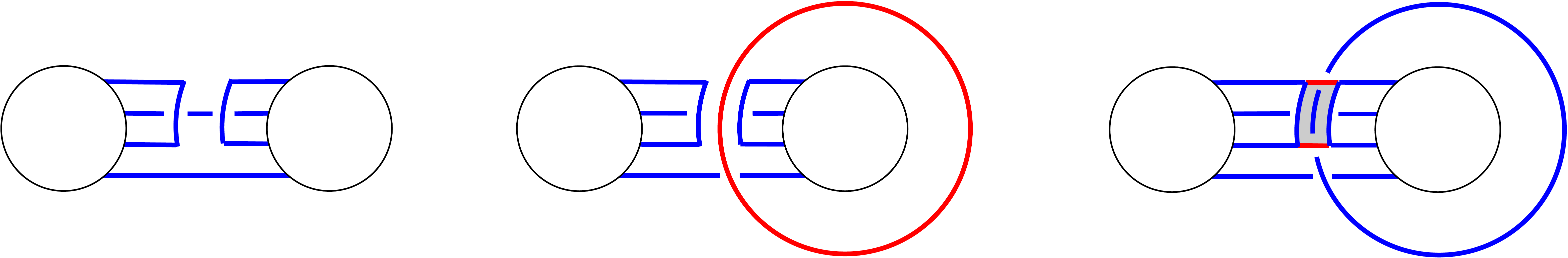}
\put(-421,33){$\cT_1$}
\put(-346,33){$\cT_2$}
\put(-279,33){$\cT_1$}
\put(-204,33){$\cT_2$}
\put(-115,33){$\cT_1$}
\put(-40,33){$\cT_2$}
\caption{
\label{fig:bandmove}
A ribbon cobordism of Euler characteristic zero from $L_1 \# L_2$ to $L$.
}
\end{figure}
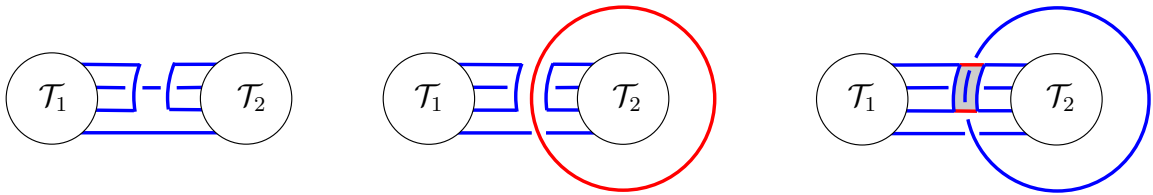

\subsubsection{A corollary on crossing number.}  The inductive construction of $\cD$ and the resolution of Tait's conjecture on crossing number \cite{kauffman,murasugi,thistle} shows that the crossing number of a nontrivial, $\ud = 1$, $\chi$-slice, prime, alternating link is congruent to $1 \pmod 3$.
For instance, this criterion shows that the pretzel $P(1,2,4,4)$, which is a nontrivial, $\ud = 1$, prime, alternating link, is not $\chi$-slice.

\subsubsection{An algorithm.} Theorem \ref{thm:links2} provides the following simple algorithm, at most quadratic in the number of crossings, to test whether a nontrivial, connected, chessboard-colored, prime alternating diagram $D$  presents a $\chi$-slice alternating link $L$ with $\ud = 1$.
We begin with an input diagram $D$ and test whether there exists a {\em decomposing circle}, i.e., a component which has no crossings with itself and four crossings with the rest of the diagram.
As we show in \Cref{cor:2cubediagram}, the set $\cD$ is closed under the inverse of the composition used to form $D$ from $D_1$ and $D_2$ above.
If a decomposing circle exists, then we decompose $D$ along it into a pair of diagrams $D_1$ and $D_2$ with $\mathrm{cr}(D_1) + \mathrm{cr}(D_2) = \mathrm{cr}(D) - 2$.
Iterate until reaching a collection of diagrams, none of which contains a decomposing circle.
Then $L$ is $\chi$-slice with $\ud = 1$ if and only if each diagram is a copy of \onecrossing\!.

\subsubsection{Complexity of rational balls.}
If $L$ is a $\chi$-slice alternating link with $\ud(L) = 1$ and $\det(L) = 4^n$, then the proof of \Cref{thm:links2} shows that $L$ bounds a $\chi$-ribbon surface $F$ with $n+1$ 0-handles and $n$ 1-handles.
The rational ball $\Sigma(F)$ therefore has a handle decomposition with 1 0-handle, $n$ 1-handles, and $n$ 2-handles (see for example \cite{gs,GLslice}).
Given how naturally this decomposition arises, one may pause to ask whether this decomposition has the minimum complexity (number of handles) amongst all rational balls filling $\Sigma(L)$.
However, this is not the case: one in a family of examples which bound a $\chi$-ribbon surface with two 0-handles and one 1-handle appears in \Cref{fig:simpleribbon}.
Still, it would be interesting to study minimum complexity rational balls filling the examples of \Cref{thm:links2}.
Which examples bound a rational ball consisting of a single 0-, 1-, and 2-handle?


\begin{figure}
\includegraphics[width=2in]{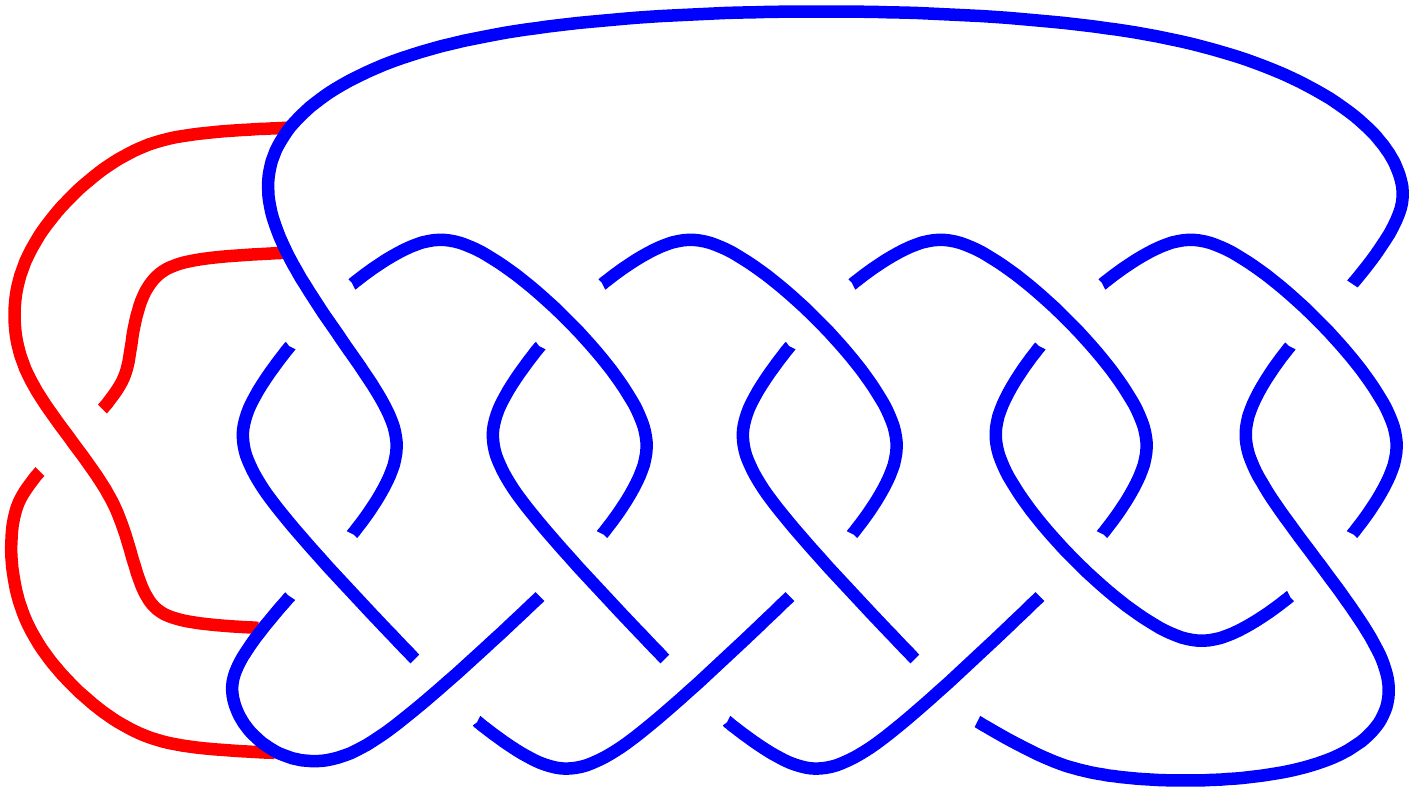}
\caption{
\label{fig:simpleribbon}
A single band move transforms the $\ud=1$ alternating link pictured into a two-component unlink.}
\end{figure}


\subsection{Ribbon cobordism and concordance.}

Our techniques extend to the finer relation of ribbon cobordism between branched double covers of alternating links.
Recall that a rational homology cobordism $W: Y_1 \to Y_2$ is a {\em ribbon cobordism} if it admits a handle decomposition rel $Y_1$ composed of 1- and 2-handles \cite{dlvvw}.
A rational homology cobordism $W: Y_1 \to Y_2$ is a {\em quasi-ribbon cobordism} if the restriction map $H^2(W;\bZ) \to H^2(Y_1;\bZ)$ surjects \cite[Definition 2.1.2]{huber:thesis}.
In this language, every ribbon cobordism is a quasi-ribbon cobordism \cite[Lemma 3.1]{dlvvw}.

Generalizing the notion of cubiquity, suppose that $\Lambda_1$ is an arbitrary lattice.
A {\em stabilization} of $\Lambda_1$ is the direct sum of $\Lambda_1$ with a Euclidean lattice of some rank.
Suppose that $\Lambda_1$ is not the stabilization of another lattice.
A sublattice $\Lambda_2 \subset \Lambda_1 \oplus \bZ^m$ is {\em cubiquitous} if every unit cube $x + \{0,1\}^m$, $x \in \Lambda_1 \oplus \bZ^m$, contains a point of $\Lambda_2$.

In what follows, let $L_1$ and $L_2$ denote a pair of non-split alternating links.

\begin{maintheorem}
\label{thm:ribboncubiquity}
Suppose that there exists a quasi-ribbon cobordism $W : \Sigma(L_1) \to \Sigma(L_2)$.
Then $\Lambda(L_2)$ admits a cubiquitous embedding into a stabilization of $\Lambda(L_1)$.
\end{maintheorem}

\noindent
\Cref{thm:ribboncubiquity} generalizes \Cref{thm:GJ1}, because puncturing a rational ball filling $\Sigma(L_2)$ yields a quasi-ribbon cobordism from $S^3 = \Sigma(\bigcirc)$ to $\Sigma(L_2)$.
The proof of \Cref{thm:ribboncubiquity} merges the proof of \Cref{thm:GJ1} with a rigidity result concerning embeddings of graph lattices into Euclidean lattices established in \cite{mut}.
We obtain a monotonicity property of $\ud$ as a result, a common theme in the study of ribbon cobordism:

\begin{maincor}
\label{cor:inequality}
If there exists a quasi-ribbon cobordism from $\Sigma(L_1)$ to $\Sigma(L_2)$, then $\ud(L_2) \le \ud(L_1)$.
\end{maincor}

Just as \Cref{thm:ribboncubiquity} generalizes \Cref{thm:GJ1}, one could make a very optimistic improvement on \Cref{conj:main}:

\begin{mainconjecture}
\label{conj:main2}
If $\Lambda(L_2)$ admits a cubiquitous embedding into a stabilization of $\Lambda(L_1)$, then there exists a ribbon cobordism from $\Sigma(L_1)$ to $\Sigma(L_2)$.
\end{mainconjecture}

\noindent Even in the case that $L_1 = \bigcirc$, \Cref{conj:main2} strengthens \Cref{conj:main}, as it asserts the existence of a {\em ribbon} rational homology ball filling $\Sigma(L_2)$ when $\Lambda(L_2)$ is cubiquitous.
The special case of \Cref{conj:main2} in which $\ud(L_1) = \ud(L_2)$ seems quite tractable.

Using \Cref{thm:ribboncubiquity}, we can classify pairs of links from \Cref{thm:links2} related by a ribbon $\chi$-concordance.
A cobordism $F : L_1 \to L_2$ in $S^3 \times [0,1]$ is a {\em $\chi$-concordance} if $\chi(F) = 0$ and $F$ has no closed components, and it is a {\em ribbon $\chi$-concordance} if, in addition, projection to $[0,1]$ defines a Morse function on $F$ with no index-2 critical points.
Note that there is no assumption placed on the orientability of $F$.
The terminology comes from \cite[Definition 2.1]{huber}, which differs somewhat from \cite[Definition 2]{DO}.
The result comes in the form of a partial order $\preceq$ on $\cD$.
Fix any chessboard colored diagram $D_0 \in \cD$.
Recursively define $D_0 \preceq \cdot$ by declaring that $D_0 \preceq D_0$ and that $D_0 \preceq D$ if $D$ is a composition of $D_1, D_2 \in \cD$ with $D_0 \preceq D_1$.


\begin{maintheorem}
\label{thm:generalization2}
Suppose that $D_1,D_2 \in \cD - \{ \round \}$ are chessboard-colored diagrams of prime, $\chi$-slice, $\ud=1$ alternating links $L_1, L_2$.
The following are equivalent:
\begin{enumerate}
\item
there exists a ribbon $\chi$-concordance $F: L_1 \to L_2$;
\item
there exists a ribbon cobordism $W: \Sigma(L_1) \to \Sigma(L_2)$;
\item
there exists a quasi-ribbon cobordism $W: \Sigma(L_1) \to \Sigma(L_2)$;
\item
$\Lambda(D_2)$ admits a cubiquitous embedding into a stabilization of $\Lambda(D_1)$; and
\item
$D_1 \preceq D_2$.
\end{enumerate}
\end{maintheorem}

We close with a pair of related conjectures:

\begin{mainconjecture}
\label{conj:alternatingribbon}
If $W : Y_1 \to Y_2$ is a ribbon rational homology cobordism and $Y_2$ is the branched double cover of an alternating link, then so is $Y_1$.
\end{mainconjecture}

\begin{mainconjecture}
\label{conj:alternatingribbon2}
If $F : L_1 \to L_2$ is a $\chi$-ribbon concordance and $L_2$ is an alternating link, then so is $L_1$.
In particular, if $F : K_1 \to K_2$ is a ribbon concordance and $K_2$ is an alternating knot, then so is $K_1$.
\end{mainconjecture}

\noindent
\Cref{conj:alternatingribbon2} would follow from \Cref{conj:alternatingribbon} and \cite[Conjecture 1.4]{mut}, which remains unsolved: if $\Sigma(L_0) \cong \Sigma(L_1)$ and $L_0$ is alternating, then so is $L_1$.

As evidence, both conjectures hold if $L_2$ a 2-bridge link, in which case $\Sigma(L_2)$ is a lens space.
In this case, the inclusions of the boundary components into $W$ induce maps $\pi_1(\Sigma(L_1)) \into \pi_1(W) \twoheadleftarrow \pi_1(\Sigma(L_2))$ \cite[Lemma 3.1]{gordon}.
Hence all groups under consideration are finite cyclic groups; hence $\Sigma(L_1)$ is a lens space, by the Geometrization Theorem; hence $L_1$ is a 2-bridge link, by a theorem of Hodgson-Rubinstein \cite[Corollary 4.12]{hr}.
We thank Marius Huber for describing this argument.


\subsection{Organization.}

Section \ref{sec:cubes} develops the basic results about 2-cube tiling lattices.
Section \ref{sec:graphs} develops the basic results about the lattice of flows on a finite graph, specializing to the case of a planar graph.
Section \ref{sec:cubegraphs} combines the results of the previous sections to characterize when the lattice of flows on a planar graph is a 2-cube tiling lattice, leading to the proof of \Cref{thm:links2}.
Section \ref{sec:cuberibbons} extends these results to prove \Cref{thm:generalization2}, assuming \Cref{thm:ribboncubiquity}.
Section \ref{sec:cubiquity} explains how the cubiquity condition issues from Heegaard Floer homology and establishes \Cref{thm:ribboncubiquity}.
Section \ref{sec:conj} describes some more examples related to this work, and also a strengthening of \Cref{conj:main}.

\subsection*{Acknowledgements}
We thank Vitalijs Brejevs for a helpful question, Marius Huber for helpful conversations, and the anonymous referee for a careful reading and corrections.  The main work in this article began in UQ\`{a}M during the CRM thematic program on Low-Dimensional Topology in 2019, and continued during the Braids semester in ICERM in 2022, and we thank these institutions for their hospitality.


\section{Lattices and cube tilings}
\label{sec:cubes}

\label{sec:tiling}
In this section, we prove \Cref{thm:tiling} and establish some basic properties of 2-cube tiling lattices and their \hajos\ matrices.
We begin with the more general notion of cubiquitous lattices.
Recall that a lattice $\Lambda=(A,q)$ is a finitely generated free abelian group $A$ together with a symmetric non-degenerate bilinear pairing $q: A \times A \to\rr$.  It is called an integral lattice if $q$ takes values in $\zz$. 
Any subgroup of $A$ gives rise to a sublattice with the restricted pairing.  A key example for us is the {\em Euclidean lattice} $\bZ^n$ of integer points in Euclidean space, with pairing given by the dot product and with orthonormal basis $e_1,\dots,e_n$.
As in the introduction, call a subset $S \subset \bZ^n$ \emph{cubiquitous} if it contains a point in every unit cube:
\[
S \cap (x + \{0,1\}^n) \ne \emptyset, \, \forall x \in \bZ^n.
\]
A lattice $\Lambda$ is cubiquitous if it admits a lattice embedding in $\zz^n$ with cubiquitous image. Such embeddings need not be unique.

\begin{prop}
\label{prop:cubiq}
The following conditions on a sublattice $\Lambda \subset \bZ^n$ are equivalent:
\begin{enumerate}
\item
$\Lambda$ is cubiquitous;
\item
every coset of $\Lambda$ is cubiquitous;
\item
every coset of $\Lambda$ contains a point of the unit cube $\{0,1\}^n$.
\end{enumerate}
\end{prop}

\begin{proof}
(1)$\implies$(2):
Suppose that $\Lambda$ is cubiquitous.
Take any coset $y + \Lambda$ and any unit cube $Q=x + \{0,1\}^n$, $x,y \in \bZ^n$.  The cube $Q'=x-y + \{0,1\}^n$ contains an element $z$ of $\Lambda$; then $y+z \in (y + \Lambda) \cap Q$.

(2)$\implies$(3): Immediate.

(3)$\implies$(1):  Let $Q=x + \{0,1\}^n$ be a unit cube.  By (3), there exists $z\in\{0,1\}^n$ in the same coset as $-x$, and then $x+z\in\Lambda\cap Q$.
\end{proof}

The {\em determinant} (a.k.a. the {\em discriminant}) of a lattice $\Lambda$ is the order of the finite group $\Lambda^*/\Lambda$, where $\Lambda^*=\Hom_\zz(\Lambda,\zz)$ denotes the dual lattice.
The determinant of a finite index sublattice of $\zz^n$ is the square of the order of the quotient $\zz^n/\Lambda$.
By \Cref{prop:cubiq}, $| \bZ^n / \Lambda | \le | \{0,1\}^n| = 2^n$ if $\Lambda$ is cubiquitous.
Hence a cubiquitous lattice of rank $n$ has determinant at most $4^n$. 
A cubiquitous sublattice $\Lambda$ of $\zz^n$ with maximal determinant $4^n$ has the property that it intersects each unit cube in a unique point.
It follows that the elements of $\Lambda$ are the centers of a tiling of $\rr^n$ by cubes of side length 2.
We therefore refer to these lattices as \emph{2-cube tiling lattices}.

One may construct such 2-cube tiling lattices recursively as follows.  There is a unique such lattice $2\zz\subset\zz$ in rank 1.  Given a 2-cube tiling lattice 
$$\Lambda_{n-1}\subset\zz^{n-1}=\{(x_1,\dots,x_n)\in\zz^n\,:\,x_1=0\},$$
we seek a 2-cube tiling lattice $\Lambda\subset\zz^n$ with $\Lambda\cap\zz^{n-1}=\Lambda_{n-1}$.
Such a lattice must contain a unique point in the cube $(1,0,\dots,0)+\{0,1\}^n$.
As $\Lambda \cap \{0,1\}^n = \{ 0 \}$, the point in question necessarily takes the form $(2,x_2,\dots,x_n)$ with each $x_i\in\{0,1\}$.
Any such point together with $\Lambda_{n-1}$ generates a cubiquitous lattice $\Lambda$.
It follows that for $\Lambda\subset\zz^n$ constructed in this way, a basis is given by the columns of a lower triangular matrix $H$ of rank $n$ with 2's on the diagonal and below-diagonal entries in $\{0,1\}$.
We refer to such a matrix as a {\em \hajos\ matrix}.
Given any \hajos\ matrix $H$, the columns of $H$ form a basis for a 2-cube tiling lattice $\Lambda_H$ in $\zz^n$.

It turns out that all 2-cube tiling lattices arise from the preceding construction:

\begin{thm}
\label{thm:tiling}
Let $\Lambda$ be a 2-cube tiling lattice in $\zz^n$.  Then up to isomorphism of $\zz^n$, $\Lambda=\Lambda_H$ for some \hajos\ matrix $H$.  If $H$ and $H'$ are \hajos\ matrices, then
$\Lambda_H=\Lambda_{H'}$ if and only if $H=H'$.
\end{thm}

We break the proof into three lemmas which will be of individual use later on.

\begin{lem}
\label{lem:hajos}
Let $\Lambda\subset\zz^n$ be a 2-cube tiling lattice.  Then the image of $\Lambda$ under an automorphism of $\zz^n$ given by reordering the orthonormal basis elements is $\Lambda_H$ for some \hajos\ matrix $H$.
\end{lem}
\begin{proof}
\hajos's theorem states that every lattice tiling of $\rr^n$ by cubes has a pair of cubes that share a codimension-1 face \cite{Hajos}.
It follows that $\Lambda$ contains $2e_i$ for some $i$; we apply a permutation of the orthonormal basis so that $h_n:=2e_n\in\Lambda$.
Complete $h_n$ to a basis $B$ for $\Lambda$, and let $M$ denote the matrix of the inclusion $\Lambda \into \zz^n$ with respect to the bases $B$ and $\{e_1,\dots,e_n\}$.
Let $\Lambda'\subset\zz^{n-1}=\{(x_1,\dots,x_n)\in\zz^n\,:\,x_n=0\}$ be the image of $\Lambda$ under projection orthogonal to $e_n$.
Let $B'$ denote the set of the projections of the elements of $B - \{h_n\}$ to $\Lambda'$.
Then $B'$ is a basis for $\Lambda'$.
Let $M'$ denote the matrix of the inclusion $\Lambda \into \zz^{n-1}$ with respect to the bases $B'$ and $\{e_1,\dots,e_{n-1}\}$.
We see that $\det M =2\det M'$, from which it follows that $\det\Lambda'=4^{n-1}$.
Also, cubiquity of $\Lambda'$ follows from that of $\Lambda$.
Hence $\Lambda'$ is a 2-cube tiling lattice.
By induction we may reorder the orthonormal basis of $\zz^{n-1}$ such that $\Lambda'$ has a basis (possibly different from $B'$) given by the columns of a \hajos\ matrix $H'$.
It follows that a basis for $\Lambda$ is given by the columns of a lower triangular matrix $H''$ with 2's on the diagonal, and whose top-left $(n-1) \times (n-1)$ submatrix equals $H'$.
By adding appropriate multiples of the last column to the remaining columns we may convert $H''$ into a \hajos\ matrix $H$ with $\Lambda= \Lambda_{H''} = \Lambda_H$.
\end{proof}

\begin{lem}
\label{lem:evenentry}
Let $\Lambda\subset\zz^n$ be a 2-cube tiling lattice, and let $x\in\Lambda$.  Then the first nonzero entry of $x$ is even.
\end{lem}
\begin{proof}
Let $H$ be a \hajos\ matrix for $\Lambda$, and let $h_1,\dots,h_n$ be the columns of $H$.  Write $x=\sum a_i h_i$, and suppose that $a_j$ is the first nonzero coefficient.  Then
$
x=(0,\dots,0,2a_j,\dots).
$
\end{proof}

We refer to a vector of the form $(0,\dots,0,2,x_{k+1},\dots,x_n)\in\zz^n$, with each $x_i\in\{0,1\}$, as a {\em \hajos\ vector}.
\begin{lem}
\label{lem:hajos=}
Let $H$ be a \hajos\ matrix of rank $n$.  Then the only \hajos\ vectors in $\Lambda_H$ are the columns of $H$.
If $H'$ is another \hajos\ matrix of rank $n$, then
$$\Lambda_H=\Lambda_{H'}\iff H=H'.$$
\end{lem}
\begin{proof}
Let $h\in\Lambda_H$ be a \hajos\ vector whose $j$-th entry is 2, and let $h_j$ be the $j$-th column of $H$.  Then $h-h_j$ is an element of $\Lambda_H$ with entries in $\{-1,0,1\}$, so by \Cref{lem:evenentry} it must be zero.  The last claim of the lemma follows immediately: if $\Lambda_H=\Lambda_{H'}$, then the columns of $H'$ are \hajos\ vectors in $\Lambda_H$.
\end{proof}

We therefore refer to the columns of $H$ as the {\em \hajos\ basis} of $\Lambda_H$.

\begin{proof}[Proof of \Cref{thm:tiling}]
This follows from Lemmas \ref{lem:hajos} and \ref{lem:hajos=}.
\end{proof}

\noindent
{\bf Remark.}
Different \hajos\ matrices $H$ and $H'$ may give rise to isomorphic lattices $\Lambda_H$ and $\Lambda_{H'}$.  This is a consequence of the fact that various choices are made in the construction of $H$ as described in the proof of \Cref{lem:hajos}: at each stage in the process one considers a 2-cube tiling lattice $\Lambda'$ and chooses a vector of the form $2e_i\in\Lambda'$ which gives rise to one of the columns of the \hajos\ matrix.  There may be more than one such vector to choose from at each stage, and in addition one could replace $e_i$ by $-e_i$ and use $-2e_i$ instead of $2e_i$.

The smallest example of this phenomenon is given by the \hajos\ matrices
\[
H=\begin{bmatrix} 2& 0& 0\\ 1& 2& 0\\ 0& 1& 2\end{bmatrix}\mbox{ and }H'=\begin{bmatrix} 2& 0& 0\\ 1& 2& 0\\ 1& 1& 2\end{bmatrix}.
\]
One obtains $H'$ from $H$ (or vice versa) by multiplying the first row and first column by $-1$, and then adding appropriate multiples of the second and third columns to the first so that the entries below the diagonal are in $\{0,1\}$.  This corresponds to a sign choice at the last step of the construction from the proof of \Cref{lem:hajos}.

\begin{lem}
\label{lem:square4}
Let $\Lambda\subset\zz^n$ be a 2-cube tiling lattice, and let $y\in\Lambda$ with $y\cdot y=4$.  Then there exists a \hajos\ matrix $H$ and an isomorphism from $\Lambda$ to $\Lambda_{H}$ taking $y$ to the last column of $H$.
\end{lem}
\begin{proof}
Let $H'$ be a \hajos\ matrix with $\Lambda=\Lambda_{H'}$.  By Lemmas \ref{lem:evenentry} and \ref{lem:hajos=}, it follows that $\pm y=2e_i$ is a column of $H'$.  
We obtain $H$ from $H'$ by transposing the $i$-th row and column with the $n$-th row and column, and then if necessary changing the sign of the last row and column of the result.
\end{proof}

\subsection{Special elements.}
\label{ss:specialhajos}
We call an element $z$ of a lattice $\Lambda$ \emph{simple} if it is not possible to write $z=x+y$ for nonzero elements $x,y$ of $\Lambda$ with $x\cdot y>0$.
We call it \emph{irreducible} (also known as \emph{indecomposable}, or a \emph{strict Voronoi vector}), if it is not possible to write $z=x+y$ for nonzero elements $x,y$ of $\Lambda$ with $x\cdot y\ge0$.
We call it \emph{rigid} if it is simple and has a unique decomposition as a sum of (two) nontrivial orthogonal elements.
For more on irreducible elements see \cite{CS,GR}.

\begin{lem}
\label{lem:hajosirred}
\hajos\ vectors in a 2-cube tiling lattice $\Lambda_H$ are irreducible.
If $h_j,h_k\in\Lambda_H$ are distinct \hajos\ vectors, then $h_j-h_k$ is irreducible if $h_j\cdot h_k > 0$ and rigid if $h_j \cdot h_k = 0$.
\end{lem}

\begin{proof}
Let $z = (z_1,\dots,z_n)\in\Lambda_H\subset\zz^n$ be a \hajos\ vector or the difference of two distinct \hajos\ vectors.  By \Cref{lem:hajos=}, one or two entries of $z$ have absolute value 2, and all other entries have absolute value 0 or 1.  Suppose that $z = x+y$ with $x,y \in \Lambda_H$.
Write $x = (x_1,\dots,x_n)$ and $y=(y_1,\dots,y_n)$.
If $|z_i|= 2$, then 
\[
x_i y_i \,
\begin{cases}
= 1,& \textup{if } |x_i|=|y_i|=1;\\
= 0,& \textup{if } \{|x_i|,|y_i|\}=\{0,2\}; \textup{and} \\
\le -1,& \textup{otherwise}.\end{cases}\]
If instead $|z_i|< 2$, then 
\[
x_i y_i \,
\begin{cases}
= 0& \textup{if } \{x_i,y_i\}=\{0,z_i\};  \\
=-1& \textup{if } \{x_i,y_i\}=\{1,-1\}; \textup{and} \\
\le -2& \textup{otherwise}. \end{cases}
\]

Suppose that $z = h_k$ is a \hajos\ vector.
We see that in order to have $x\cdot y\ge0$ we must have $\{|x_k|,|y_k|\}=\{0,2\}$ or $\{1,1\}$ and all other entries of $x$ and $y$
have absolute value at most one.
Hence one of $x$ and $y$ has no nonzero even entry, so \Cref{lem:evenentry} implies it is 0.
It follows that \hajos\ vectors are irreducible.

Suppose instead that $z=h_j-h_k$ is a difference of distinct \hajos\ vectors.
Without loss of generality, $j < k$.
Then $z_j = 2$, $z_k \in \{-1,-2\}$, and all other entries $z_i$ are in $\{-1,0,1\}$.

If $z_k = -1$, then $h_j \cdot h_k > 0$ and the previous argument shows that $z$ is irreducible.
This gives one of the desired outcomes.

We henceforth assume that $z_k = -2$ and establish the other outcome.
A similar application of  \Cref{lem:evenentry} in this case shows that $x\cdot y\ge0$ is possible only if $\{|x_j|,|y_j|\}=\{|x_k|,|y_k|\}=\{0,2\}$ and $\{x_i,y_i\}=\{0,z_i\}$ for $i\notin\{j,k\}$, showing that in fact $x$ and $y$ have disjoint support and $x\cdot y=0$.
We may assume without loss of generality, again using \Cref{lem:evenentry}, that $x_j = 2$ and $y_k = -2$.  Moreover these are the first nonzero entries of both $x$ and $y$.

Write $x = \sum_i a_i h_i$ in terms of the columns of the \hajos\ matrix $H$.
The vectors $x$ and $h_j$ agree in every coordinate up through the $k$-th.
Since $H$ is lower triangular with non-zero diagonal, it follows that $a_j = 1$ and otherwise $a_i = 0$ for $i \le k$.
Thus, we have $x = h_j + \sum_{i > k} a_i h_i$.
Similarly, writing $y$ in the \hajos\ basis, its coefficient on $h_k$ is $-1$ and its coefficient on $h_i$ is $0$ for $i < k$.
Substituting these expressions for $x$ and $y$ in $h_j - h_k = x+y$ yields $y = -h_k - \sum_{i>k} a_i h_i$.
If there exists a value $i > k$ for which $a_i \ne 0$, then let $m$ denote the smallest such.
The $m$-th entries of $x$ and $h_j$ are in $\{-1,0,1\}$ and differ by $2a_m \ne 0$, and the same applies to the $m$-th entries of $y$ and $h_k$.
It follows that all entries in question are $\pm 1$; but then $m \in \supp(x) \cap \supp(y)=\emptyset$, a contradiction.
Consequently no such value $m$ exists, and we infer that $x = h_j$ and $y = -h_k$.
Hence $h_j \cdot h_k = 0$ and $h_j - h_k$ is rigid, as required.
\end{proof}

\subsection{Indecomposable summands.}
\label{ss:indecomp}
A lattice $\Lambda$ is {\em indecomposable} if any orthogonal direct sum decomposition $\Lambda = \Lambda_1 \oplus \Lambda_2$ implies that $\Lambda_1 = 0$ or $\Lambda_2 = 0$.
Eichler proved that every lattice admits an orthogonal direct sum decomposition into indecomposable lattices, unique up to order of factors \cite[Theorem II.6.4]{mh}.
If $B$ is a basis of $\Lambda$ consisting of irreducible elements, then we can describe this decomposition in the following way.
Put an equivalence relation $\sim$ on $B$ by declaring that $b \sim b'$ if there exists a sequence of elements $b = b_1, \dots, b_k = b'$ in $B$ such that $b_i \cdot b_{i+1} \ne 0$ for all $i = 1, \dots, k-1$.
Then the equivalence classes of $B$ under $\sim$ generate the indecomposable summands of $\Lambda$.
For a proof, see \cite[Section 3.1]{lens}.

\begin{cor}
\label{cor:cubedecompose}
If $\Lambda = \Lambda_1 \oplus \Lambda_2$, then $\Lambda$ is a 2-cube tiling lattice if and only if both of $\Lambda_1$ and $\Lambda_2$ are.
Hence the indecomposable summands of a 2-cube tiling lattice are 2-cube tiling lattices. 
\end{cor}

\begin{proof}
Let $\Lambda$ be a 2-cube tiling lattice with \hajos\ matrix $H$.
By the preceding paragraph and \Cref{lem:hajosirred}, a decomposition $\Lambda = \Lambda_1 \oplus \Lambda_2$ corresponds to a partition $B_1 \sqcup B_2$ of the columns of $H$ such that $h_1 \cdot h_2 = 0$ for all $h_i \in B_i$, $i=1,2$.
Permute the columns of $H$ so that those of $B_1$ precede all those of $B_2$, preserving the order within each $B_i$, and apply the same permutation to its rows.
Doing so puts $H$ into block diagonal form $H_1 \oplus H_2$, where $H_1$ and $H_2$ are \hajos\ matrices, and $\Lambda_i \cong \Lambda_{H_i}$, $i=1,2$.

For the other implication note that cubiquity is preserved under orthogonal direct sum, and the additivity properties of rank and determinant.
\end{proof}


\section{Lattices and flows on graphs}
\label{sec:graphs}

In this section we study the lattice of integer-valued flows on a finite graph, with a focus on planar graphs.
More details and relevant graph theory appear in \cite{bdlhn,GR,mut}.  

\subsection{Flows and cuts on finite graphs.}

Let $G$ be a finite graph, and fix an arbitrary orientation of the edges of $G$.
Doing so gives $G$ the structure of a CW complex and produces an affiliated cellular chain complex
\[
0 \to C_1(G;\bZ) \overset{\del}{\to} C_0(G;\bZ) \to 0.
\]
The chain groups $C_0(G;\bZ)$ and $C_1(G;\bZ)$ inherit the structures of Euclidean lattices by declaring the vertex set and the oriented edge set to form orthonormal bases $V$ and $E$ for these groups, respectively.
The \emph{flow lattice} of $G$ is the sublattice
\[
\flow(G) := \ker(\del) \subset C_1(G;\bZ).
\]
Thus, $\flow(G)$ is none other than $H_1(G;\bZ)$ equipped with a positive definite inner product.
Informally, an element of $\flow(G)$ assigns a direction and a weight in $\zz_{\ge0}$ to each edge of $G$ such that at each vertex, the sum of the incoming weights equals the sum of the outgoing weights.

With respect to the given orthonormal bases, we obtain an adjoint map
\[
0 \to C_0(G;\bZ) \overset{\del^*}{\to} C_1(G;\bZ) \to 0.
\]
The {\em cut lattice} of $G$ is the sublattice
\[
\cut(G) := \im(\del^*) \subset C_1(G;\bZ).
\]
The lattices $\flow(G)$ and $\cut(G)$ are complementary, primitive sublattices of $C_1(G;\bZ)$ \cite[Proposition 3.1]{mut}.

An oriented subgraph $H  \subset G$ determines an element $[H] \in C_1(G;\bZ)$ by the rule that for each $e \in E$, the pairing $[H] \cdot e$ equals $+1$, $-1$, or $0$ depending on whether $e$ is an edge of $H$, its reversal $\overline{e}$ is an edge of $H$, or neither $e$ nor $\overline{e}$ is an edge of $H$.
We distinguish a few important cases: the subgraph $H$ is

\begin{itemize}
\item
an {\em oriented Eulerian subgraph} if $[H] \in \flow(G)$.
Equivalently, $H$ has an equal number of incoming and outgoing edges at each vertex.
\item
an {\em oriented cycle} if $[H] \in \flow(G)$, $H$ is connected, and each vertex of $H$ has a single incoming and outgoing edge.
We typically write $H = C$ in this case.
\item
a {\em oriented cut} if $[H] \in \cut(G)$.
Equivalently, $H$ is the set of all edges $E(A,B)$ directed from a vertex in a subset $A \subset V$ to its complement $B=V - A$; that is, $[H] = \del^* \sum_{v \in A} v$.
\item
a {\em oriented bond} if $[H] \in \cut(G)$ and, in the notation above, the subgraphs of $G$ induced on the vertex sets $A$ and $V-A$ are connected.
\item
an {\em orientation} if $[H] \cdot e = \pm 1$, $\forall e \in E$.
We typically write $H = O$ in this case.
\end{itemize}

The {\em support} of $x \in C_1(G;\bZ)$ is the set $\supp(x) := \{ e \in E : x \cdot e \ne 0 \}$; similarly, $\supp^+(x) = \{e \in E : x \cdot e > 0 \}$, and $\supp^-(x) = \{e \in E : x \cdot e < 0 \}$.
Let $\Lambda \subset C_1(G;\bZ)$ denote either sublattice $\flow(G)$ or $\cut(G)$.
An element $y \in \Lambda$ is {\em primitive} if its support is inclusion-minimal amongst all non-zero elements of $\Lambda$ and, in addition, $y \cdot e \in \{-1,0,+1\}$, $\forall e \in E$.
Tutte proved that $\Lambda$ is a {\em regular} lattice: this means that for every $x \in \Lambda$, $x \ne 0$, there exists a primitive element $y$ such that $\supp^+(y) \subset \supp^+(x)$ and $\supp^-(y) \subset \supp^-(x)$.
Moreover, $y \in \flow(G)$ is primitive if and only if it is the class of an oriented cycle, and $y \in \cut(G)$ is primitive if and only if it is the class of a oriented bond \cite{tutte}.

\subsection{Special elements in graph lattices.}
\label{ss:specialflow}
We now determine the indecomposable summands and simple, irreducible, and rigid elements of a graph lattice; refer to the beginnings of \Cref{ss:specialhajos} and \Cref{ss:indecomp} for the definitions.

We leave the proof of the first result as an exercise.
While it is phrased as a statement about graphs, it is, of course, just a statement about Euclidean lattices.

\begin{prop}
\label{prop:Znirred}
In the lattice $C_1(G;\bZ)$,
\begin{enumerate}[(a)]
\item
the indecomposable summands are the sublattices $(e)$, where $e \in E$;
\item
the simple elements are the classes of oriented subgraphs of $G$;
\item
the irreducible elements are the elements $\pm e$, where $e \in E$; and
\item
the rigid elements are the elements $\epsilon_1 e_1 +\epsilon_2 e_2$, where $e_1, e_2 \in E$ are distinct and $\epsilon_i\in\{\pm1\}$. 
\end{enumerate}
\end{prop}

A graph is {\em 2-connected} if it is connected and does not contain a cut-vertex or a loop.
The {\em blocks} of a finite graph $G$ are its maximal induced 2-connected subgraphs.
A block is {\em trivial} if it consists of a single vertex or edge and {\em nontrivial} otherwise.

\begin{prop}
\label{prop:flowirred}
In the lattice $\flow(G)$,
\begin{enumerate}[(a)]
\item
the indecomposable summands are the sublattices $\flow(B)$, where $B$ is a nontrivial block of $G$;
\item
the simple elements are the classes of oriented Eulerian subgraphs of $G$;
\item
the irreducible elements of $\flow(G)$ are the classes of oriented cycles of $G$; and
\item
the rigid elements are the classes of oriented Eulerian subgraphs of $G$ which are the edge-disjoint union of a pair of oriented cycles that intersect in at most one vertex.
\end{enumerate}
\end{prop}

\noindent
Statements (a), (b), and (c) should be considered well-known and can be deduced from the references \cite{bdlhn,GR,mut}, but we elaborate on them here for completeness.

\begin{proof}
(a) Since each cycle in $G$ is contained in a unique non-trivial block $B$ of $G$, it is clear that $\flow(G)$ decomposes as a sum of the sublattices $\flow(B)$.
The indecomposability of each summand $\flow(B)$ then follows from \cite[Proposition 4]{bdlhn}.

(b) 
Suppose that $x \in \flow(G)$.
Since $\flow(G)$ is regular, we can write $x = y_1 + \cdots + y_k$ as a sum of primitive elements, where $\supp^+(y_i) \subset \supp^+(x)$ and $\supp^-(y_i) \subset \supp^-(x)$ for all $i=1,\dots,k$.
Suppose that $|x \cdot e_0| > 1$ for some $e_0 \in E$.
Then in the preceding expression, we can locate some $y_i$ for which $y_i \cdot e_0 \ne 0$.
Then $x = y_i + (x-y_i)$, and the condition on supports shows that $(y_i \cdot e)((x-y_i) \cdot e) \ge 0$ for all $e \in E$, with strict inequality if $e = e_0$.
Summing over all $e \in E$ gives $y_i \cdot (x-y_i) > 0$, so $x$ is not simple.
It follows that if $x$ is simple, then $|x \cdot e| \le 1$ for all $e \in E$.

Conversely, suppose that $|x \cdot e| \le 1$ for all $e \in E$.
Write $x = y + z$ with $y,z \in \flow(G)$.
Then $(y \cdot e) + (z \cdot e) = x \cdot e \in \{-1,0,+1\}$, $\forall e \in E$.
This implies that $(y \cdot e) (z \cdot e) \le 0$, $\forall e \in E$, whence $y \cdot z \le 0$, so $x$ is simple.

(c) From the preceding argument, we see that if $x$ is simple, $x = y + z$, and $y \cdot z = 0$, then $\supp(x) = \supp(y) \sqcup \supp(z)$.
This is possible if and only if $y=0$, $z=0$, or $x$ is not primitive.
Hence $x$ is irreducible if and only if $x$ is the class of an oriented cycle.

(d) Suppose now that $x$ is simple, $x = y + z$, $y \cdot z = 0$, and $y,z \ne 0$.
Then $x$, $y$, and $z$, are simple, so $x = [H]$, $y = [H_1]$, and $z = [H_2]$ for some non-empty Eulerian subgraphs $H$, $H_1$, and $H_2$ with $E(H) = E(H_1) \sqcup E(H_2)$.

If $H=C_1\cup C_2$ is the edge-disjoint union of a pair of oriented cycles that intersect in at most one vertex, then it follows that $C_1 = H_1$ and $C_2 = H_2$, up to reordering the factors.
Hence, in this case, $[H]$ is rigid.

Conversely, if $[H]$ is rigid, then each of $H_1$ and $H_2$ is a cycle (possibly plus some isolated vertices).
Suppose by way of contradiction that the cycles intersect in more than one vertex.
Then we can locate such a pair $v$ and $w$ so that the oriented path from $v$ to $w$ along $H_1$ does not contain any vertices of the cycle in $H_2$ in its interior.
Take the union of this path with the oriented path from $w$ to $v$ in $H_2$ to create an oriented cycle $C$ in $H$ distinct from $H_1$ and $H_2$.
Then $[H] = [C] + [H - C]$ is a decomposition into orthogonal nonzero elements distinct from $[H] = [C_1] + [C_2]$, contradicting the rigidity of $[H]$.
Therefore, the cycles intersect in at most one vertex.
\end{proof}

We leave the proof of the following result as an exercise, as well, since it is formally identical to the proof of \Cref{prop:flowirred}(a)-(c).
Note that we do not state a description of the rigid elements of $\cut(G)$, as we do not need it here.

\begin{prop}
\label{prop:cutirred}
In the lattice $\cut(G)$,
\begin{enumerate}[(a)]
\item
the indecomposable summands are the sublattices $\cut(B)$, where $B$ is a nontrivial block of $G$;
\item
the simple elements of $\cut(G)$ are the classes of oriented cuts in $G$; and
\item
the irreducible elements of $\cut(G)$ are the classes of oriented bonds in $G$. \qed
\end{enumerate}
\end{prop}

The next result combines the preceding descriptions of simple elements.
It is required to prove \Cref{prop:uniqueshort}.

\begin{lem}
\label{lem:subgraph}
Suppose that $H$ is an oriented subgraph of $G$.
Then $[H] \in \Flow(G) \oplus \Cut(G) \subset C_1(G;\bZ)$ if and only if $[H] = [H_1] + [H_2]$, where $H_1$ is an Eulerian subgraph, $H_2$ is a oriented cut, and $H_1$ and $H_2$ have disjoint edge sets.
\end{lem}

\noindent Here is another way to phrase \Cref{lem:subgraph}, in light of the preceding results.
Suppose that $x_1 \in \flow(G)$ and $x_2 \in \cut(G)$.
Then $x_1 + x_2 \in C_1(G;\bZ)$ is simple if and only if $x_1$ is simple in $\flow(G)$, $x_2$ is simple in $\cut(G)$, and $\supp(x_1) \cap \supp(x_2) = \emptyset$.

\begin{proof}
The reverse direction is immediate from the definitions, so we prove the forward direction.  We may assume that $G$ is connected, since otherwise one may simply work with one connected component at a time.
We proceed by induction on the number of edges in $H$.
Write $[H] = x + y$, where $x \in \Flow(G)$ and $y \in \Cut(G)$.
Then we may write $y = \del^*(z)$ for some $z = \sum_{v \in V} z_v \cdot v \in C_0(G;\bZ)$.
Let $m = \max \{z_v : v \in V \}$, let $A = \{ v \in V : z_v = m \}$, and let $B = V(G) - A$.
If $B = \emptyset$, then $y = 0$, and $H$ is an oriented Eulerian subgraph, as desired.
Hence we may assume that $B \ne \emptyset$, and so the oriented cut $E(A,B)$ is non-empty, as well.
For every oriented edge $e = (v,w) \in E(A,B)$, we have $[H] \cdot e \le 1$, with equality if and only if $e$ is an oriented edge in $H$.
On the other hand, $y \cdot e = z_v - z_w \ge m - (m-1) = 1$ for each such edge.
Therefore, if $[E(A,B)] \in \Cut(G)$ denotes the class of the cut, then it follows that
\[
[E(A,B)] \cdot [E(A,B)] \ge [H] \cdot [E(A,B)] = y \cdot [E(A,B)] \ge [E(A,B)] \cdot [E(A,B)],
\]
using the fact that $x \in \Flow(G) = \Cut(G)^\perp$ in the middle equality.
Since all inequalities are equalities, it follows that $E(A,B)$ is an oriented subgraph of $H$.
Let $H'$ denote the subgraph of $H$ obtained by discarding the edges of $E(A,B)$.
Since $E(A,B) \ne \emptyset$, $H'$ has fewer edges than $H$.
Since $[H'] = [H] - [E(A,B)] \in \Flow(G) \oplus \Cut(G)$, it follows by induction that $H'$ decomposes as a disjoint union of an Eulerian subgraph $H_1'$ and an oriented cut $H_2'$ in $G$.
The union $H_2$ of $H_2'$ and $E(A,B)$ is an oriented subgraph with $[H_2] = [H_2'] + [E(A,B)] \in \Cut(G)$.  Simplicity of $[H_2]$ follows from that of $[H]$, so $H_2$ is an oriented cut, by \Cref{prop:cutirred}(b).
Therefore, $H$ admits the required decomposition with $H_1 = H_1'$ and $H_2$, and the proof is complete.
\end{proof}

\subsection{Flows on planar graphs.}

A planar graph is a graph $G$ with a fixed drawing on $S^2$ without edge crossings.
A set of cycles $C_1,\dots,C_n$ in $G$ is {\em nested} if each $C_i$ is the boundary of a disk $D_i \subset S^2$ and, up to reordering, $D_1 \subset \cdots \subset D_n$.
Note that the definition does not take orientations of cycles, boundaries, or $S^2$ into consideration.

\begin{lem}
\label{lem:pwsnest}
Let $G$ be a planar graph and let $C_1$ and $C_2$ be oriented cycles in $G$.
Suppose that
\begin{enumerate}
\item
$[C_1] \cdot [C_2] < 0$ and $[C_1] + [C_2]$ is irreducible, or
\item
$[C_1] \cdot [C_2] = 0$ and $[C_1] + [C_2]$ is rigid.
\end{enumerate}
Then $C_1$ and $C_2$ are nested, $[C_1] \cdot [C_2] = -|E(C_1 \cap C_2)|$, and $C_1 \cap C_2$ is either connected or empty.
\end{lem}

\begin{proof}
Fix a planar drawing of $G$ on the sphere $S^2$.
The oriented cycles $C_1$ and $C_2$ bound oriented disks $D_1$ and $D_2$ in $S^2$, respectively, which we regard as 2-chains with coefficients 0 and 1 (say, with respect to a cellular decomposition of $S^2$ whose 1-skeleton includes $G$).
The sum $D_1+D_2$ is a 2-chain with coefficients in $\{0, 1, 2\}$ and with boundary $[C_1]+[C_2]$.

First suppose that case (1) holds.
By Proposition \ref{prop:flowirred}(b), $[C_1]+[C_2]$ is the class of an oriented cycle $C$, so it bounds an oriented disk $D$.
Moreover, the 2-chains that $C$ bounds take the form $D + n \cdot [S^2]$, $n \in \bZ$.
The only such 2-chains with coefficients in $\{0, 1, 2\}$ are $D$ and $D + [S^2]$.
The only decompositions of $D$ as a sum of disks $D_1+D_2$ occur when $D_1$ and $D_2$ have disjoint interiors and $C_1 \cap C_2$ is a path containing at least one edge.
Put another way, $C_1 \cup C_2$ is homeomorphic to a theta-graph, and $D_1$ and $D_2$ are two of its lobes.
Similarly, the only decompositions of $D + [S^2]$ as a sum of disks $D_1 + D_2$ occurs when $D_1$ and $D_2$ have disjoint {\em exteriors} and $C_1 \cap C_2$ is a path containing at least one edge.
Once more, $C_1 \cup C_2$ is homeomorphic to a theta-graph, but now $D_1$ and $D_2$ are complementary to two of its lobes.
In either case, we obtain the conclusion of the lemma.

Next suppose that case (2) holds.
By Proposition \ref{prop:flowirred}(d), $C_1$ and $C_2$ intersect in at most one vertex.
Since $G$ is planar, the conclusion of the lemma follows once more.
\end{proof}

The next result asserts a Helly property for cycles in a planar graph.

\begin{lem}
\label{lem:nested}
Suppose that $G$ is a planar graph and that $C_1,\dots,C_n$ are pairwise nested oriented cycles in $G$ such that $[C_i] \cdot [C_j] > 0$ for all $1 \le i,j \le n$.
Then $C_1,\dots,C_n$ are nested and share at least one edge in common.
\end{lem}

\begin{proof}
The proof is by induction on $n$.
The statement is immediate for $n = 2$, so suppose that $n > 2$ and the statement holds for fewer than $n$ such cycles.
By the induction hypothesis, the set of edges $E$ in common to $C_1,\dots,C_{n-1}$ is non-empty, and we may relabel the cycles so that $C_i = \del D_i$ for disks $D_i \subset S^2$ satisfying $D_1 \subset \cdots \subset D_{n-1}$.
Furthermore, we let $D_i' \subset S^2$ denote the complementary disk with $D_i \cap D_i' = C_i$; thus, $D_{n-1}' \subset \cdots \subset D_1'$.
Because $C_n$ is nested with each of $C_1,\dots,C_{n-1}$, it follows that $C_n \subset D_i$ or $C_n \subset D_i'$ for each $i$.

Suppose that $C_n$ bounds a disk $D_n$ that contains $D_{n-1}$.
Then $C_1,\dots,C_n$ are nested.
Since $[C_1] \cdot [C_n] > 0$, the set of edges in $C_1 \cap C_n$ is non-empty, and since $D_1 \subset \cdots \subset D_n$, it follows that $C_1 \cap C_n \subset C_1 \cap \cdots \cap C_n$, so the set of edges in common to these cycles is non-empty.
If instead $C_n$ bounds a disk $D_n$ contained in $D_1$, then a similar argument guarantees the desired conclusion.

We are left to consider the possibility that $C_n$ is contained in $D_k$ but not $D_{k-1}$ for some value $1 < k < n-1$.
Let $D_n$ denote the disk that $C_n$ bounds that is contained in $D_k$.
If $D_n$ does not contain $D_{k-1}$, then $D_{k-1}'$ contains $D_n$, and the interiors of $D_{k-1}$, $D_k'$, and $D_n$ are pairwise disjoint.
But then the triple product $([C_{k-1}] \cdot [C_k]) ([C_k] \cdot [C_n]) ([C_n] \cdot [C_{k-1}])$ is not positive, in contradiction to the assumption that each term is positive.
Consequently, $D_n$ contains $D_{k-1}$, and we infer that $C_1,\dots,C_n$ are nested.
As $D_{k-1} \subset D_n \subset D_k$ and $E \subset C_{k-1} \cap C_k$, it follows that $E \subset C_n$ as well, so $\emptyset \ne E \subset C_1 \cap \cdots \cap C_n$ certifies that the cycles contain a common edge.
\end{proof}

\subsection{Flows and the characteristic coset.}
\label{ss:char}

The material of this subsection and the next is not needed until \Cref{sec:cubiquity}.
We begin by recalling some notions from \cite[Section 2.2]{mut}.

Let $\Lambda$ denote a positive definite, integral lattice.
Denote by
\[
\Char(\Lambda) = \{ \chi \in \Lambda^* : \chi \cdot \lambda \equiv \lambda \cdot \lambda \, (\textup{mod } 2), \, \forall \lambda \in \Lambda \}
\]
the set of {\em characteristic elements} of $\Lambda$.
It is a coset of $2 \Lambda^*$ in $\Lambda^*$.
The set
\[
\cX(\Lambda) := \Char(\Lambda) / 2 \Lambda
\]
is a torsor of the determinant group $\Lambda^* / \Lambda$.
Denote by
\[
\Short(\Lambda) = \{ \chi \in \Char(\Lambda) : \chi \cdot \chi \le (\chi - 2 \lambda) \cdot (\chi - 2\lambda), \forall \lambda \in \Lambda \}
\]
the set of {\em short} characteristic elements of $\Lambda$.
Thus, $\chi \in \Short(\Lambda)$ if and only if $\chi$ has the minimum self-pairing of any representative of its coset in $\cX(\Lambda)$.

As an example, in terms of any orthonormal basis for the Euclidean lattice $\bZ^n$, $\Char(\bZ^n)$ consists of all elements with odd coefficients, and $\Short(\bZ^n) = \{\pm 1 \}^n$.
Thus, for any graph $G$, the elements of $\Short(C_1(G;\bZ))$ are the classes of the $2^n$ orientations on $G$.

The embedding $\flow(G) \into C_1(G;\bZ)$ induces a restriction map of dual lattices $C_1(G;\bZ)^* \to \flow(G)^*$.
As shown in \cite[Corollary 3.4]{mut}, the restriction map induces a {\em surjective} map $\Short(C_1(G;\bZ)) \to \Short(\flow(G))$.
In concrete terms, if $\chi \in \Short(\flow(G))$ is the restriction of the class of an orientation $O$ on $G$, and if $C$ is an oriented cycle in $G$, then $\chi \cdot [C]$ is equal to the signed sum of edges in $C$, where the sign of an edge is $+1$ or $-1$ according to whether or not its orientation in $C$ matches its orientation in $O$.

An orientation on $G$ is {\em acyclic} if it does not contain the edge set of any oriented cycle in $G$.
Equivalently, an orientation on $G$ is acyclic if and only if it is possible to order the vertices of $G$ in such a way that all edges direct from a lower vertex to a higher vertex.

\begin{prop}
\label{prop:uniqueshort}
An element $\chi \in \Short(\flow(G))$ is the unique short element in its coset in $\cX(\flow(G))$ if and only if it is the restriction of the class of an acyclic orientation on $G$.
\end{prop}

\begin{proof}
Suppose that $O_1$ and $O_2$ are orientations on $G$ and the restrictions of their classes define elements $\chi_1$ and $\chi_2$ in $\Short(\flow(G))$ that represent the same coset in $\cX(\flow(G))$.
Thus, we may write $\chi_1 = \chi_2 + 2 x$ with $x \in \Flow(G)$.
It follows that $[O_1] - ([O_2] + 2x) \in \Flow(G)^\perp = \Cut(G) \subset C_1(G;\bZ)$.
On the other hand, $[O_1] - [O_2] = 2[H]$, where $H$ is the oriented subgraph of $G$ consisting of oriented edges in $O_1$ which appear with the opposite orientation in $O_2$.
Since $\Cut(G) \subset C_1(G;\bZ)$ is primitive, we can write $2[H] - 2 x = 2 y$ with $y \in \Cut(G)$.
Hence $[H] = x + y$ with $x \in \Flow(G)$ and $y \in \Cut(G)$.
By \Cref{lem:subgraph}, it follows that $x = [H_1]$ and $y = [H_2]$, where $H_1$ is an Eulerian subgraph, $H_2$ is a oriented cut, and $H_1$ and $H_2$ have disjoint edge sets.
Let $O_1'$ denote the orientation obtained from $O_1$ by changing the directions on all of the edges in $H_2$.
Then $[O_1] - [O_1'] = 2[H_2]$ and $[O_1'] - [O_2] = 2[H_1]$.
The first identity implies that $[O_1]$ and $[O_1']$ both restrict to $\chi_1 \in \Short(\flow(G))$, and
the second implies that $O_1'$ and $O_2$ differ along the edges of the Eulerian subgraph $H_1$.

Now suppose that $O_1$ is acyclic.
Changing the directions on the edges of a oriented cut produces another acyclic orientation, so $O_1'$ is acyclic as well.
It follows that $E(H_1)$ is empty, so $O_1' = O_2$.
Therefore, $\chi_1 = \chi_2$ is unique in its equivalence class.

Suppose instead that $O_1$ contains a directed cycle $C$.
Let $O_2$ be the result of changing the directions of the edges on $C$ in $O_1$.
Then $[O_1]-[O_2] = 2 [C]$ implies that $\chi_1$ and $\chi_2$ are distinct short elements representing the same coset in $\cX(\flow(G)$).
\end{proof}

Likewise, an orientation is {\em strongly connected} if it does not contain the edge set of an oriented cut.
The same argument shows that an element of $\Short(\cut(G))$ is unique in its coset in $\cX(\cut(G))$ if and only if it is the restriction of the class of a strongly connected orientation on $G$.


\subsection{Flows and the $d$-invariant.}
\label{ss:dinvt}

The $d$-invariant of a positive definite integral lattice records the self-pairings of its short characteristic elements \cite[Sections 2.4]{mut}.
More precisely, we define a function
\[
d_\Lambda : \cX(\Lambda) \to \bQ, \quad d_\Lambda(s) := \min \left\{ \frac{\chi \cdot \chi - \rk(\Lambda)}4 : \chi \in s \right\}.
\]
The pair $(\cX(\Lambda),d_\Lambda)$ is the {\em $d$-invariant of $\Lambda$}.

If $\Lambda_1$ and $\Lambda_2$ are two positive definite, integral lattices, then an {\em isomorphism }of their $d$-invariants 
is a bijection $\varphi: \cX(\Lambda_1) \to \cX(\Lambda_2)$ covering a group isomorphism $\Lambda_1^* / \Lambda_1 \overset{\sim}{\to} \Lambda_2^* / \Lambda_2$ with the additional property that $d_{\Lambda_1} = d_{\Lambda_2} \circ \varphi$.

The following result is implicit in \cite[Theorem 3.8]{mut}, but we make it explicit here and supply a careful argument.

\begin{prop}
\label{prop:rigid}
Suppose that $\Lambda$ is a lattice whose $d$-invariant is isomorphic to that of $\flow(G)$, where $G$ denotes a 2-edge-connected graph.
Then $\Lambda$ is isometric to a stabilization of $\flow(G)$.
\end{prop}

\noindent
We remark that both the result and its proof hold with the roles of $\cut(G)$ and $\flow(G)$ reversed, but we shall not need it here.

\begin{proof}
By \cite[Corollary 3.4]{mut}, there exists an isomorphism $\psi$ between {\em minus} the $d$-invariant of $\cut(G)$ and the $d$-invariant of $\flow(G)$.
Compose $\psi$ with an isomorphism to the $d$-invariant of $\Lambda$ to obtain an isomorphism $\varphi $ between minus the $d$-invariant of $\cut(G)$ and the $d$-invariant of $\Lambda$.
By \cite[Proposition 2.7]{mut}, we can glue $\cut(G)$ and $\Lambda$ using $\varphi$ to obtain an embedding of $\cut(G) \oplus \Lambda$ into a Euclidean lattice $\bZ^m$ with the property that the images of $\cut(G)$ and $\Lambda$ are complementary, primitive sublattices of $\bZ^m$ and the restriction maps $\Short(\bZ^m) \to \Short(\cut(G))$ and $\Short(\bZ^m) \to \Short(\Lambda)$ surject.
By \cite[Lemma 3.2]{mut} and the hypothesis that $G$ is 2-edge-connected, the embedding $\cut(G) \into C_1(G;\bZ) =: Z_1$ satisfies part of the hypothesis of \cite[Proposition 2.8]{mut}, while we have just shown that the embedding $\cut(G) \into \bZ^m =: Z_2$ satisfies the complementary part.
Therefore, \cite[Proposition 2.8]{mut} shows that the embedding $\cut(G) \into \bZ^m$ is the composite of the embedding $\cut(G) \into C_1(G;\bZ)$ with an embedding $C_1(G;\bZ) \into \bZ^m$.
The second embedding is simply a stabilization of $C_1(G;\bZ)$, since both are Euclidean lattices.
It follows that the orthogonal complement to the image of $\cut(G)$ in $\bZ^m$ is isometric to a stabilization of the orthogonal complement to the image $\cut(G)$ in $C_1(G;\bZ)$, which is $\flow(G)$, by \cite[Proposition 3.1]{mut}.
But we already identified the orthogonal complement of the embedding of $\cut(G)$ into $\bZ^m$ with $\Lambda$, based on the remark that they are complementary and primitive in $\bZ^m$.
It follows that $\Lambda$ is isometric to a stabilization of $\flow(G)$, as required.
\end{proof}


\section{Cube tilings and flows on graphs}
\label{sec:cubegraphs}

The goal of this section is to characterize which planar graphs $G$ have the property that $\flow(G)$ is a 2-cube tiling lattice (\Cref{thm:2cubeflow}).
It leads directly to the characterization of which chessboard-colored alternating link diagrams $D$ have the property that $\Lambda(D)$ is a 2-cube tiling lattice (\Cref{cor:2cubediagram}).
Throughout the section, we abuse notation and write $\flow(G)=\Lambda_H$ when we have fixed an isomorphism between the two lattices.
We correspondingly use equality to identify elements of $\flow(G)$ with their images in $\Lambda_H$.

\hajos's theorem implies that every 2-cube tiling lattice of positive rank contains an irreducible element of square 4.
If the lattice is also the flow lattice of a graph, then \Cref{prop:flowirred} implies that this element corresponds to a 4-cycle $C$.

\begin{prop}
\label{prop:planarhajos}
Let $G$ be a planar graph such that $\flow(G)$ is a 2-cube tiling lattice of rank $n>0$, let $C$ be an oriented 4-cycle in $G$, and let $R_1$ and $R_2$ be the regions $C$ bounds in $S^2$.
Then there exist
\begin{itemize}
\item
a \hajos\ matrix $H$;
\item
oriented cycles $C_1,\dots,C_n \subset G$; and
\item
a cyclic ordering $e_1,\dots,e_4$ of the edges of $C$
\end{itemize}
such that
\begin{enumerate}
\item
$\flow(G)=\Lambda_H$;
\item
$[C_i]=h_i$ is the $i$-th column of $H$ for $i=1,\dots,n$; 
\item
$C_n = C$; and
\item
for each $i < n$, $C_i$ is contained in one of the regions $R_j$, and either
\begin{itemize}
\item
the last entry of $h_i$ is 1 and $C_i \cap C_n = e_j \cup e_{j+1}$, or else
\item
the last entry of $h_i$ is 0 and $C_i \cap C_n$ is a single vertex or empty.
\end{itemize}
\end{enumerate}

\end{prop}

\begin{figure}[t]
\centering
\includegraphics[scale=0.6]{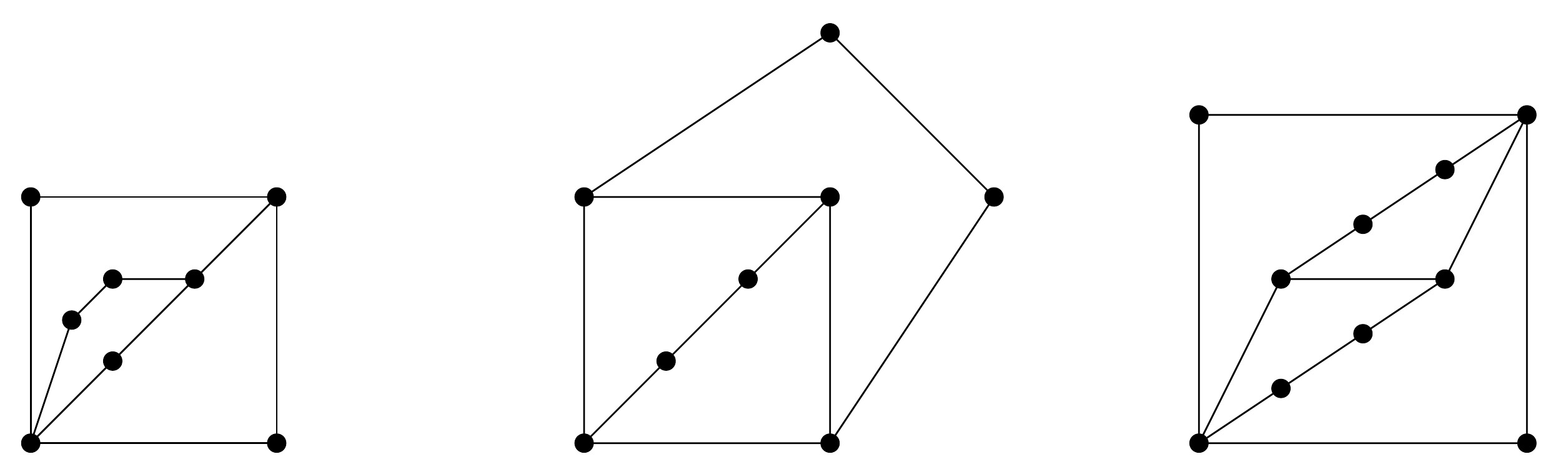}
\put(-450,100){$G_1$}
\put(-300,100){$G_2$}
\put(-140,100){$G_3$}
\put(-450,-30){$H_1=\begin{bmatrix} 2& 0& 0\\ 1& 2& 0\\ 1& 1& 2\end{bmatrix}$}
\put(-300,-30){$H_2=\begin{bmatrix} 2& 0& 0\\ 0& 2& 0\\ 1& 1& 2\end{bmatrix}$}
\put(-140,-40){$H_3=\begin{bmatrix} 2& 0& 0& 0\\ 0& 2& 0&0\\ 1& 1& 2& 0\\ 1& 1& 1& 2\end{bmatrix}$}
\caption
{
\label{fig:planarexamples}
Some planar graphs $G_i$ with 2-cube tiling flow lattices $\Lambda_{H_i}$, $i=1,2,3$.}
\end{figure}
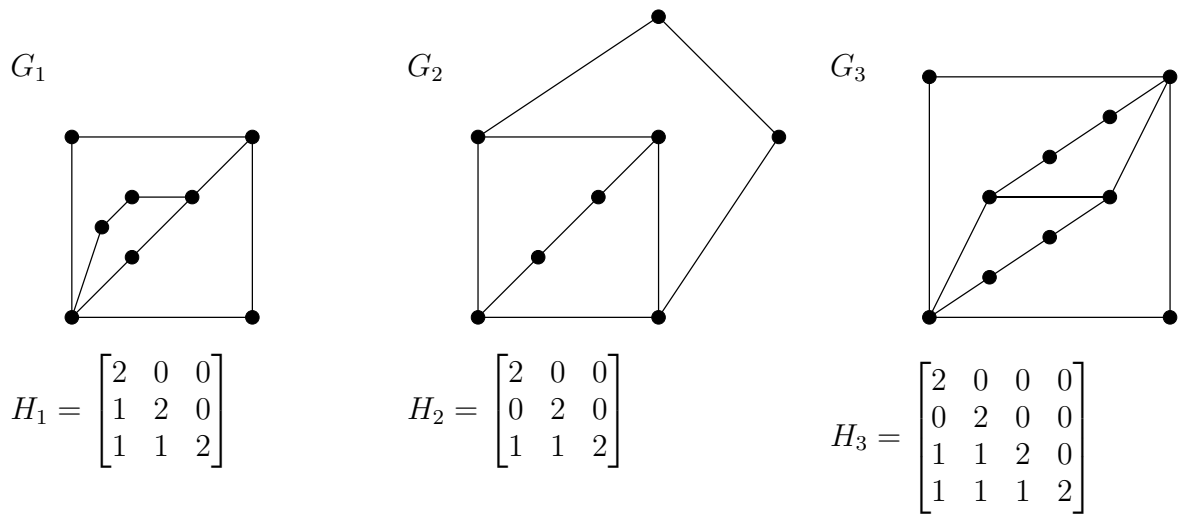

\begin{proof}
By \Cref{lem:hajos}, $\flow(G)$ is isomorphic to $\Lambda_H$ for some \hajos\ matrix $H$.  
By \Cref{lem:hajosirred}, the \hajos\ basis elements $h_1,\dots,h_n$ are irreducible.
By \Cref{prop:flowirred}, they therefore take the form $h_i=[C_i]$ for some oriented cycles $C_1,\dots,C_n$ in $G$.
By \Cref{lem:square4}, we may suppose that $C_n=C$.
These remarks establish (1)-(3).

By \Cref{lem:hajosirred}, $[C_i] - [C_j] = h_i - h_j$ is either irreducible or rigid, for all $i \ne j$.
Hence \Cref{lem:pwsnest} implies that the cycles $C_1,\dots,C_n$ are pairwise nested and $|E(C_i \cap C_j)| = [C_i] \cdot [C_j]$ for all $i$ and $j$.
In addition, $|E(C_i \cap C_n)|= h_i \cdot h_n \in \{0,2\}$ depending on whether the last entry of $h_i$ is 0 or 1.
Another look at \Cref{lem:pwsnest} gives all of (4) apart from the statement that $C_i \cap C_n = e_j \cup e_{j+1}$ when $C_i \subset R_j$.

Let $I$ consist of those $i\in \{1,\dots,n-1\}$ with $h_i \cdot h_n = [C_i]\cdot [C_n]>0$, or equivalently with the last entry of $h_i$ equal to 1.
By \Cref{lem:nested}, the cycles $\{C_i\,:\,i\in I\cup\{n\}\}$ are nested and contain a common edge $e' \in E(C_n)$.  
We partition $I = I_1 \sqcup I_2$, where $I_\ell$ consists of those $i \in I$ such that $C_i \subset R_\ell$ for $\ell=1,2$.

By \Cref{lem:pwsnest}, $C_i \cap C_n$ is a path of length $[C_i] \cdot [C_n] = h_i \cdot h_n = 2$ for all $i \in I$.
Select a pair of distinct indices $i,j \in I$.
If $C_i$ and $C_j$ are contained in the same region $R_\ell$, then because $C_i$ and $C_j$ are nested, it follows that $C_i \cap C_n = C_j \cap C_n$ is the same path of length 2.
This means that $C_n$ contains edges $e'_1$ and $e'_2$, consecutive with $e'$ along $C_n$, such that $C_i \cap C_n = \{e',e'_\ell\}$ for all $i \in I_\ell$ and $\ell=1,2$.

We claim that if $I_1, I_2 \ne \emptyset$, then $e'_1 \ne e'_2$.
(Note that if one of $I_1$ or $I_2$ is empty, then we may vacuously take $e'_1 \ne e'_2$.)
For suppose that $C_i \subset R_1$ and $C_j \subset R_2$ and, for a contradiction, that $e'_1 = e'_2$.
It follows that $h_i \cdot h_j = [C_i] \cdot [C_j] = |E(C_i \cap C_j)| = |\{e',e'_1\}| = 2$.
Hence $h_i$ and $h_j$ both have $k$-th entry equal to 1 for some $i,j<k<n$.
The cycle $C_k$ is contained in the opposite region $R_\ell$ from one of $C_i$ and $C_j$; without loss of generality, say it is $C_i$.
Then the intersection of any pair of $C_i$, $C_k$, and $C_n$ is $\{e',e'_1\}$.
This implies that $h_k \cdot h_n = 2$, so the last entry of $h_k$ equals 1.
Considering the $k$-th and $n$-th entries of $h_i$ and $h_k$, this implies that $h_i \cdot h_k \ge 1 \cdot 2 + 1 \cdot 1 = 3$, in contradiction of the fact that $h_i \cdot h_k = |E(C_i) \cap E(C_k)| = 2$.
Thus, $e'_1 \ne e'_2$.

Lastly, we choose the cyclic ordering of the edges of $C$ by taking $e_1 = e'_1$, $e_2 = e'$, and $e_3 = e'_2$.
This completes the proof of (4) and hence the proposition.
\end{proof}

Suppose that $G$ is a planar graph.
A {\em decomposing cycle} in $G$ is a 4-cycle $C$ such that
\begin{itemize}
\item
the vertices of $C$ are $v_1,v_2,v_3,v_4$ in cyclic order;
\item
the edges of $C$ are $e_1,e_2,e_3,e_4$ in cyclic order;
\item
$C$ divides $S^2$ into regions $R_1$ and $R_2$; and
\item
vertex $v_i$ has degree 2 in the subgraph $G \cap R_j$ whenever $i \not\equiv j \pmod 2$.
\end{itemize}
Thus, $G$ takes the form shown on the left side of Figure \ref{fig:planarhajos}.

\begin{figure}
\centering
\includegraphics[width=4in]{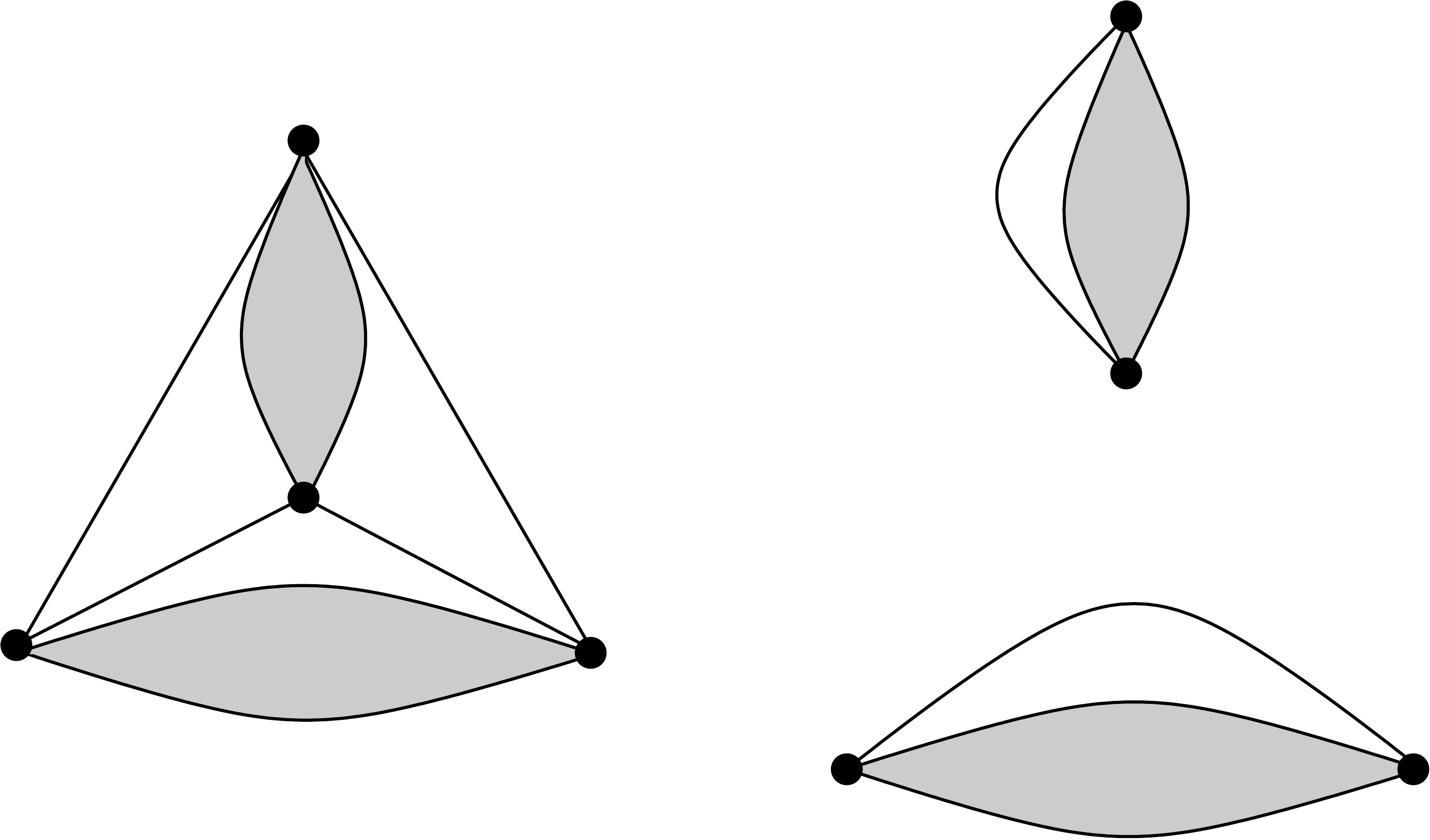}
\put(-300,130){$G$}
\put(-233,100){$J_1$}
\put(-233,35){$J_2$}
\put(-233,150){$v_1$}
\put(-305,40){$v_2$}
\put(-233,57){$v_3$}
\put(-163,40){$v_4$}
\put(-280,90){$e_1$}
\put(-260,65){$e_2$}
\put(-210,65){$e_3$}
\put(-185,90){$e_4$}
\put(-125,165){$G_1$}
\put(-68,125){$J_1$}
\put(-68,175){$v_1$}
\put(-68,82){$v_3$}
\put(-105,125){$f_1$}
\put(-120,50){$G_2$}
\put(-68,10){$J_2$}
\put(-140,15){$v_2$}
\put(3,15){$v_4$}
\put(-30,45){$f_2$}
\caption
{\label{fig:planarhajos}
A planar graph $G$ with a decomposing cycle and its factors $G_1$ and $G_2$.
}
\end{figure}
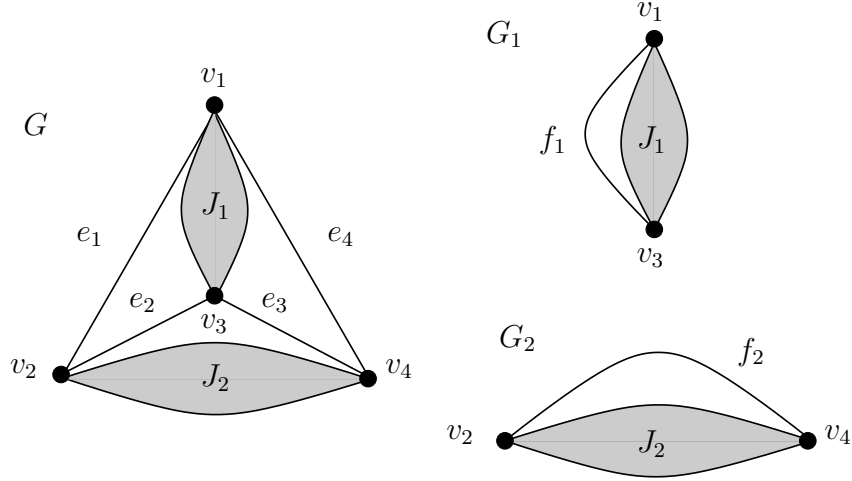

\begin{cor}
\label{cor:IImut}
Let $G$ be a 2-connected planar graph such that $\flow(G)$ is a 2-cube tiling lattice.
Then every 4-cycle in $G$ is a decomposing cycle.
\end{cor}

\begin{proof}
Let $C$ be a 4-cycle in $G$ with vertices $v_1,v_2,v_3,v_4$ and edges $e_1=(v_1, v_2),e_2,e_3,e_4$ in cyclic order.
Apply \Cref{prop:planarhajos}.
The flow lattice of $G \cap R_1$ is generated by classes of cycles, each of which uses $e_2$ if it uses $e_1$.
It follows that every cycle in $G \cap R_1$ uses $e_2$ if it uses $e_1$.
Suppose for a contradiction that the degree of $v_2$ in $G \cap R_1$ is greater than 2, and choose an edge $e=(v_2,v)$ different from $e_1$ and $e_2$.
Since $G$ is 2-connected, so is $G \cap R_1$.
It follows that there exists a path $P$ from $v$ to $v_1$ in $G \cap R_1$ avoiding $v_2$.
Then the union of $P$ with the edges $e_1$ and $e$ forms an oriented cycle in $G$ which meets $C$ in the edge $e_1$ but not $e_2$, a contradiction.
It follows that $v_2$ has degree 2 in $G \cap R_1$, and a similar argument shows that $v_i$ has degree 2 in $G \cap R_j$ whenever $i \not\equiv j \pmod 2$.
Hence $C$ is a decomposing cycle, as desired.
\end{proof}

Let $G$ be a planar graph with a decomposing cycle $C$.
Form a planar graph $G_1$ out of the subgraph $G \cap R_1$ by contracting $e_2$ and deleting $e_3$, $v_4$, and $e_4$.
Similarly, form a planar graph $G_2$ out of the subgraph $G \cap R_2$ by contracting $e_2$ and deleting $e_1$, $v_1$, and $e_4$.
We call $G_1$ and $G_2$ the {\em factors} of $G$ {\em decomposed along} $C$.
Again, see \Cref{fig:planarhajos}.

\begin{lem}
\label{lem:decompose}
Suppose that $G$ is a planar graph, $C$ is a decomposing cycle in $G$, and $G_1$ and $G_2$ are the factors of $G$ decomposed along $C$.
\begin{enumerate}
\item
$G$ is 2-connected if and only if $G_i$ is 2-connected, $i=1,2$.
\item
$\flow(G)$ is a 2-cube tiling lattice if and only if $\flow(G_i)$ is a 2-cube tiling lattice, $i=1,2$.
\end{enumerate}
\end{lem}

\begin{proof}
(1) A planar graph is disconnected if and only if there exists a circle in the plane which is disjoint from the graph and contains at least one vertex of the graph to either side.
A planar graph contains a cut-vertex or a loop if and only if there exists a circle in the plane whose intersection with the graph is a single vertex and contains at least one edge of the graph to either side.
It is easy to check that the existence of such a circle for $G$ implies one for either $G_1$ or $G_2$, and vice versa.
This establishes the statement about 2-connectivity.

(2) Suppose first that $\flow(G)$ is a 2-cube tiling lattice.
Apply \Cref{prop:planarhajos}.
Note that for $i < n$, the last entry of $h_i$ is 1 if $C_i$ contains the edge $e_2$ and 0 if not.
Also, $C_n$ is the only cycle amongst $C_1,\dots,C_n$ which contains the edge $e_4$.
Let $G'$ denote the graph obtained by deleting the edge $e_4$ and contracting the edge $e_2$ in $G$.
Since the classes $[C_1],\dots,[C_n]$ generate $\flow(G)$, it follows that the classes $[C_1'],\dots,[C_{n-1}']$ generate $\flow(G')$, where $C_i' \subset G'$ denotes the image of the cycle $C_i$.
Let $H'$ denote the \hajos\ matrix obtained by deleting the final row and column from $H$.
Let $h_i'$ denote the $i$-th column of $H'$.
Then it is direct to check that $h_i' \cdot h_j' = [C_i'] \cdot [C_j']$ for all $1 \le i,j < n$.
Hence $\flow(G') = \Lambda_{H'}$ is a 2-cube tiling lattice.
Moreover, $G_1$ and $G_2$ are isomorphic to the blocks of $G'$, since each is 2-connected and they meet at a cut-vertex of $G'$.
Hence $\flow(G') = \flow(G_1) \oplus \flow(G_2)$, and \Cref{cor:cubedecompose} implies that $\flow(G_i)$ is a 2-cube tiling lattice, $i=1,2$.

Conversely, suppose that $\flow(G_i)$ is a 2-cube tiling lattice, $i=1,2$.
Write $\flow(G_i) = \Lambda_{H_i}$ for a \hajos\ matrix $H_i$, $i=1,2$ by \Cref{prop:planarhajos}.
Form the block sum $H' = H_1 \oplus H_2 \oplus (2)$.
Let $h_i$ be a column of $H_1$ or $H_2$.
Then $h_i = [C_i]$ for a cycle $C_i$ in $G_1$ or $G_2$.
Change the last entry in the $i$-th column of $H'$ to a 1 if $C_i$ contains the distinguished edge of the graph it is in.
Then it is direct to check the resulting matrix $H$ is a \hajos\ matrix which satisfies $\flow(G) = \Lambda_H$.
\end{proof}

Decomposition along a decomposing cycle is reversible, as in \Cref{fig:planarhajos}.
In words, suppose we are given a pair of planar graphs $G_1$ and $G_2$ with distinguished non-loop edges $f_1 = (v_1,v_3) \in E(G_1)$ and $f_2 = (v_2,v_4) \in E(G_2)$.
Draw $G_1 \sqcup G_2$ on $S^2$ in such a way that $f_1$ and $f_2$ abut to the same region.
Remove $f_1$ and $f_2$ and put a 4-cycle $C$ on $(v_1,v_2,v_3,v_4)$ with edges $(e_1,e_2,e_3,e_4)$ to obtain the planar graph $G$.
We call $G$ a {\em composition} of $G_1$ and $G_2$.
Observe that $G$ depends on the drawings of $G_1$ and $G_2$ on $S^2$, even after distinguishing the edges $f_1$ and $f_2$.
For instance, we get different compositions by rotating either drawing of $J_i = G_i \setminus \{f_i \}$, $i=1,2$, by $180^\circ$.  Note also that the number of edges of a composition is equal to the sum of the numbers of edges of its factors, plus 2.

Define a set of planar graphs $\cG$ as follows.
Let \singleton denote the graph with one vertex and no edges, and let \edge denote the graph with two vertices and one edge between them.
Recursively define $\cG$ by declaring that \singleton, \edge $\in \cG$ and that $G \in \cG$ whenever $G$ is a composition of two graphs $G_1, G_2 \in \cG$.

\begin{thm}
\label{thm:2cubeflow}
Let $G$ be a finite planar graph.
Then $\flow(G)$ is a 2-cube tiling lattice if and only if the blocks of $G$ belong to $\cG$.
Moreover, $G$ is 2-connected with $\flow(G)$ a 2-cube tiling  lattice if and only if $G \in \cG$.
\end{thm}

\begin{proof}
We proceed by induction on the number of edges.
The base case of 0 or 1 edges is clear, so we proceed on to the induction step.

Suppose that $\flow(G)$ is a 2-cube tiling lattice.
\Cref{prop:flowirred}(a) asserts that $\flow(G)$ decomposes as a direct sum of indecomposable sublattices $\flow(B)$, where $B$ is a block of $G$.
\Cref{cor:cubedecompose} implies that $\flow(B)$ is a 2-cube tiling lattice for each block $B$.
If $B$ is acyclic, then $B = \text{\singleton}$ or $B = \text{\edge}$, hence $B \in \cG$.
Otherwise, $\flow(B)$ has positive rank, so the remark before \Cref{prop:planarhajos} implies that it contains a 4-cycle $C$.
\Cref{cor:IImut} implies that $C$ is a decomposing cycle.
Let $B_1$ and $B_2$ denote the factors of $B$ decomposed along $C$.
Both of $B_1$ and $B_2$ have fewer edges than $B$, and \Cref{lem:decompose} implies that $B_i$ is 2-connected and $\flow(B_i)$ is a 2-cube tiling lattice, $i=1,2$.
By induction, $B_1, B_2 \in \cG$, so $B \in \cG$, as well.
This completes the induction step.
Note that if $G$ is 2-connected, then the argument shows that $G = B \in \cG$.

The converse direction follows from reversing the argument.
\end{proof}

We now translate \Cref{thm:2cubeflow} into the language of link diagrams.
Recall that the Tait graph construction establishes a one-to-one correspondence between connected chessboard-colored alternating link diagrams and connected planar graphs.
Let $D$ be such a diagram and $G(D)$ its black Tait graph.
Thus, $\Lambda(D) = \flow(G(D))$.
Under the correspondence, nugatory crossings in $D$ which touch a single white region are in one-to-one correspondence with cut-edges in $G(D)$, while nugatory crossings in $D$ which touch a single black region are in one-to-one correspondence with loops in $G(D)$.
Hence $D$ is reduced if and only if $G(D)$ is loopless and does not contain a cut-edge, i.e., no block of $G(D)$ contains a single edge.
Similarly, reducing circles in $D$ are in one-to-one correspondence with circles in the plane meeting $G(D)$ in a vertex and containing at least one edge to either side.
Hence the prime summands of $D$ are in one-to-one correspondence with the blocks of $G(D)$.

For the translation into link diagrams, we refer to the set $\cD$ from the introduction.

\begin{cor}
\label{cor:2cubediagram}
Suppose that $D$ is a chessboard-colored connected alternating diagram.
Then $\Lambda(D)$ is a 2-cube tiling lattice if and only if $D$ is a connected sum of diagrams in $\cD$.
In addition, if $\Lambda(D)$ is a 2-cube tiling lattice then $D$ is prime if and only if $D \in \cD$, and $D$ is reduced if and only if it is a connected sum of diagrams in $\cD \setminus \{ \text{\onecrossing}\!\! \}$.
Lastly, if $D$ is a composition of $D_1$ and $D_2$, then $D \in \cD$ if and only if $D_1$ and $D_2 \in \cD$.
\end{cor}

\begin{proof}
Under the Tait graph correspondence, $G($\round$\!\!) = \text{\singleton}$, $G(\text{\onecrossing}\!\!) = \text{\edge}$, and $D$ is a composition of $D_1$ and $D_2$ if and only if $G(D)$ is a composition of $G(D_1)$ and $G(D_2)$.
Hence $D \in \cD$ if and only if $G(D) \in \cG$.

Write $D = D_1 \# \cdots \# D_k$ as a connected sum of prime diagrams.
Then $G(D)$ has blocks $G(D_1), \cdots, G(D_k)$.
Hence $D$ is a connected sum of diagrams in $\cD$ if and only if each $G(D_i)$ belongs to $\cG$.
\Cref{thm:2cubeflow} now gives the initial statement of the Corollary.
Since $D$ is prime if and only if $G(D)$ is 2-connected, and since $D$ is reduced if and only if no block of $G(D)$ consists of a single edge, \Cref{thm:2cubeflow} gives the next statement as well.
The final statement comes from \Cref{lem:decompose}.
\end{proof}

\Cref{thm:2cubeflow} leads to a simple algorithm to test whether a planar graph $G$ has the property that $\flow(G)$ is a 2-cube tiling lattice.
First, decompose $G$ into its blocks.
Next, given a block, test whether it contains a 4-cycle.
If so, test whether the cycle is a decomposing cycle, and if so, decompose along it to get a pair of planar graphs $G_1$ and $G_2$ with $|E(G_1)| + |E(G_2)| = |E(G)| - 2$.
Iterate until reaching a graph with a 4-cycle that is not a decomposing cycle or a collection of graphs none of which contains a 4-cycle.
Then $\flow(G)$ is a 2-cube tiling lattice if and only if the resulting collection of graphs are copies of \singleton and \edge.

The algorithm of the introduction simply translates this one into the language of link diagrams.


\section{Cubes and ribbons}
\label{sec:cuberibbons}

In this section, we extend the techniques developed in \Cref{sec:cubegraphs} to prove \Cref{thm:generalization2}.
The bulk of the argument is contained in the following result:

\begin{prop}
\label{prop:2cube}
Suppose that $D_1$ and $D_2$ are reduced alternating diagrams of $\chi$-slice, $\ud=1$ alternating links $L_1$ and $L_2$, where $D_1 \in \cD \setminus \{ \round \}$.
Suppose that $\Lambda(L_2)$ admits a cubiquitous embedding into a stabilization of $\Lambda(L_1)$.
Then $D_1 \preceq D_2'$ for some prime summand $D_2'$ of $D_2$.
\end{prop}

\begin{proof}
Let $\Lambda_1$ denote the image of a cubiquitous embedding of $\Lambda(L_1)$ into $\bZ^m$.
Similarly, let $\Lambda_2$ denote the image of a cubiquitous embedding of $\Lambda(L_2)$ into $\Lambda_1 \oplus \bZ^n$.
The condition that $\Lambda_2 \subset \Lambda_1 \oplus \bZ^n$ is cubiquitous and and $\ud(\Lambda_1) = \ud(\Lambda_2)$ imply that
\begin{center}
($\star$) each cube $\lambda + x + \{0,1\}^n$, $\lambda + x \in \Lambda_1 \oplus \bZ^n$, contains a unique point of $\Lambda_2$.
\end{center}
Thus, $\Lambda_2$ is akin to a 2-cube tiling lattice of $\Lambda_1 \oplus \bZ^n$.

We proceed to build a \hajos\ basis for $\Lambda_2\subset\Lambda_1\oplus\bZ^n\subset\bZ^{m+n}$.
First, since $\Lambda_1$ is a 2-cube tiling lattice of $\bZ^m$, we can choose a \hajos\ matrix $H_1 = (h_1' \cdots h_m')$ for $\Lambda_1$.
By ($\star$), there exist unique elements $x_1,\dots,x_m \in \{0,1\}^n$ such that $h_i := h_i'+x_i \in \Lambda_2$ for $i=1,\dots,m$.
Next, applying ($\star$) to each point of $(0) \oplus \bZ^n$, we find that $\Lambda_2 \cap (0) \oplus \bZ^n$ is a 2-cube tiling lattice.
Thus, we can find a \hajos\ matrix $H_2 = (h_{m+1} \cdots h_{m+n})$ for $\Lambda_2 \cap (0) \oplus \bZ^n$.
Now we see that $H = (h_1 \cdots h_{m+n})$ is a \hajos\ matrix of $\Lambda_2$.
Its top left $m \times m$ matrix is $H_1$, and its bottom right $n \times n$ matrix is $H_2$.

By \Cref{prop:planarhajos}, the element $h_{m+n} \in \Lambda_2$ corresponds to a decomposing cycle $C$ in the Tait graph $G_2$ of $D_2$.
Let $G_{21}$ and $G_{22}$ denote the factors of $G_2$ decomposed along $C$, with corresponding diagrams $D_{21}$ and $D_{22}$.
Thus, $D_2$ is a composition of $D_{21}$ and $D_{22}$, each of which is a connected sum of diagrams in $\cD$.
Let $D_3$ denote a color-respecting connected sum of $D_{21}$ and $D_{22}$, and let $L_3$ denote the link it presents.
The top-left $(m+n-1) \times (m+n-1)$ matrix of $H$ is a \hajos\ matrix $H_3$ for $\Lambda_3 \simeq \Lambda(L_3)$.
Thus, $L_3$ is a $\ud=1$, $\chi$-slice alternating link, and $\Lambda_3$ is a 2-cube tiling lattice of $\Lambda_1 \oplus \bZ^{n-1}$.
If $n=1$ then in fact $D_1$ is a reduced diagram of $L_3$, and we conclude that one of $D_{21}$ and $D_{22}$ is $\text{\edge}$ and $D_2$ is a composition of $D_1$ and $\text{\edge}$.

If $n>1$ then by induction, we have $D_1 \preceq D'_3$ for some prime summand $D_3'$ of $D_3$.
Either $D_3'$ is itself a prime summand of $D_2$; or else $D_3'$ is a prime summand of one of $D_{21}$ and $D_{22}$, and it composes with a prime summand of the other to form a prime summand of $D_2$.
In either case, $D_1 \preceq D_2'$ for some prime summand $D_2'$ of $D_2$,  as desired.
\end{proof}

\begin{proof}
[Proof of \Cref{thm:generalization2}]
(1)$\implies$(2) follows from taking the branched double cover.

(2)$\implies$(3) was already noted \cite[Lemma 3.1]{dlvvw}.

(3)$\implies$(4) follows from \Cref{thm:ribboncubiquity}.

(4)$\implies$(5) is \Cref{prop:2cube}, specialized to the case in which $D_2$ is prime.

(5)$\implies$(1) follows just as in \Cref{thm:links2}, (5)$\implies$(1).
\end{proof}


\section{Cubiquity}
\label{sec:cubiquity}

In this section we explain how cubiquity issues from Heegaard Floer homology.

\subsection{Cubiquity and rational homology balls}

Let $Y$ be a rational homology 3-sphere, that is to say, a closed 3-manifold with $b_1(Y)=0$.
Recall from \cite{absgr} that the Heegaard Floer homology of \ozsvath\ and \szabo\ gives rise to a correction term invariant
\begin{align*}
d:\spinc(Y)&\longto\qq\\
\spinct&\longmapsto d(Y,\spinct).
\end{align*}
We will use two properties of this invariant.
First, if $Y$ bounds a rational homology ball $W$ and $\spinct \in \spinc(Y)$ extends to a \spinc\ structure on $W$, then $d(Y,\spinct) = 0$.
Second, if $Y$ bounds a positive-definite 4-manifold $X$ and $\spinct$ extends to a \spinc\ structure $\spincs$ on $X$, then
\begin{equation}
\label{eq:d}
d(Y,\spinct) \ge \frac{c_1(\spincs)^2-b_2(X)}4.
\end{equation}
(In fact, the second statement generalizes the first.)
We call $\spincs$ a {\em sharp extension} of $\spinct$ if it attains equality in \Cref{eq:d}, and we call
$X$ {\em sharp}
if each \spinc structure on $Y$ admits a sharp extension to $X$.

The lattice $d$-invariant is designed to model the $d$-invariant of the intersection lattice on a positive definite 4-manifold.
Keeping the notation above, suppose that the restriction from $\spinc(X)$ to $\spinc(Y)$ surjects, and let $\Lambda_X=(H_2(X;\zz)/\tors,Q_X)$ denote the intersection lattice of $X$.
The image of the map $c_1 : \Spc(X) \to H^2(X;\bZ)/\tors \simeq \Lambda_X^*$ is the set $\Char(\Lambda_X)$.
The group $H^2(Y;\bZ)$ is isomorphic with the determinant group $\Lambda_X^* / \Lambda_X$, and there is a torsor isomorphism $\varphi: \Spc(Y) \to \Char(\Lambda_X) / 2 \Lambda_X$ which covers the group isomorphism.
Inequality \eqref{eq:d} then amounts to the inequality $d(Y,\spincs) \ge d_{\Lambda_X}(\varphi(\spincs))$ for all $\spincs \in \Spc(Y)$, with equality if and only if $X$ is sharp.

\noindent
{\bf Example.}
Let $L$ be a non-split alternating link, and let $D$ be a reduced chessboard-colored alternating diagram of $L$.
The black regions of $D$, along with half-twisted bands at the crossings, form a spanning surface $B \subset S^3$ of $L$.
Push the interior of $B$ into that of a 4-ball bounded by $S^3$, and let $X = \Sigma(B)$ denote the double cover of the 4-ball branched along $B$.
Gordon and Litherland showed that the intersection lattice of $X$ is isometric to $\Lambda(L)$ \cite{GL}, and Ozsv\'ath and Szab\'o showed that $X$ is a sharp 4-manifold \cite{HFdbc}.
We call $X$ the {\em Gordon-Litherland manifold} filling $\Sigma(L)$.
(Strictly speaking, $X$ depends on the choice of diagram $D$ used to present $L$, but it can be shown that its diffeomorphism type is independent of this choice.)

The following obstruction to a rational homology 3-sphere bounding both a sharp 4-manifold and a rational homology ball follows quickly from work by the first author and Jabuka \cite{GJ}.  We give a proof here for the convenience of the reader, and because it sets up the proof of \Cref{thm:ribboncubiquity}.
\begin{thm}
\label{thm:GJ}
Let $Y$ be a rational homology 3-sphere which bounds both a sharp 4-manifold $X$ and a rational homology 4-ball $W$.  Then the intersection lattice of $X$ is cubiquitous.
\end{thm}

\begin{proof}
Form the smooth, closed, positive-definite 4-manifold $Z=X\cup_Y W$.
The Mayer-Vietoris sequence and Donaldson's Theorem A \cite{thmA} yield an embedding $\Lambda_X \into \Lambda_Z \cong \zz^n$.
We will show that the image is cubiquitous.

We have a map $c : \spinc(Z) \to \Lambda_Z^*$ by taking the first Chern class modulo torsion.
The image of $c$ equals the set $\mathrm{Char}(\Lambda_Z)$.
Sharpness of $X$ implies that the restriction map from $\spinc(X)$ to $\spinc(Y)$ surjects, from which it follows that two \spinc structures $\spincs_1$ and $\spincs_2$ on $Z$ have the same restriction to $Y$ if and only if $c(\spincs_1) - c(\spincs_2) \in 2\Lambda_X$.  
Thus, we obtain an identification between the set of $\spc$ structures on $Y$ which extend over $Z$ with $\mathrm{Char}(\bZ^n) / 2 \Lambda_X$.

Select any $\spinct \in \spinc(Y)$ which extends over $Z$, and let $\spincs \in \spinc(Z)$ denote an extension.
Then $\spincs|W$ is an extension of $\spinct$ to $W$, so $d(Y,\spinct) = 0$.
Select a sharp extension $\spincs_X$ of $\spinct$ to $X$.
Then $\spincs_X$ and $\spincs|W$ glue to give a new $\spinc$ structure $\spincs' \in \spinc(Z)$ which restricts to $\spinct$ on $Y$, and it satisfies $c_1(\spincs')^2 -n= c_1(\spincs_X)^2-n = 4d(Y,\spinct) = 0$.
Consequently, $c(\spincs') \in \{ \pm 1 \}^n$.
Thus, every coset of $\mathrm{Char}(\zz^n)/2\Lambda_X$ contains a representative in $\Short(\bZ^n) = \{\pm1\}^n$.

Lastly, we have a bijection
\[
\mathrm{Char}(\zz^n)/2\Lambda_X\overset{\sim}\longrightarrow\zz^n/\Lambda_X
\]
given by $[\xi]\mapsto[(\xi-\xi_0)/2]$, with $\xi_0=(-1,\dots,-1)$.
It follows that every coset of $\zz^n/\Lambda_X$ contains a representative in $\{0,1\}^n$.
Therefore, $\Lambda_X$ is cubiquitous by \Cref{prop:cubiq}.
\end{proof}

The following corollary issues immediately from \Cref{thm:GJ} and the example preceding it.

\begin{cor}
\label{cor:altcubiquity}
Let $L$ be a non-split alternating link.
If $\Sigma(L)$ bounds a rational homology ball, then $\Lambda(L)$ is cubiquitous. \qed
\end{cor}

\subsection{Cubiquity and quasi-ribbon cobordisms}
Next, we look to generalize \Cref{thm:GJ} to the case of a quasi-ribbon cobordism.
We begin with a simple and very useful lemma:

\begin{lem}
\label{lem:hereditarysharp}
Suppose that $W : Y_1 \to Y_2$ is a quasi-ribbon cobordism and $X$ is a sharp 4-manifold filling $Y_2$.
Then $W \cup X$ is a sharp 4-manifold filling $Y_1$.
\end{lem}

\begin{proof}
First, $W \cup X$ is definite, since $b_2(W) = 0$ and $X$ is definite.
Next, select any $\spinct_1 \in \Spc(Y_1)$.
Since $H^2(W) \to H^2(Y_1)$ surjects, so does $\Spc(W) \to \Spc(Y_1)$.
Hence there exists $\spincs_W \in \Spc(W)$ which restricts to $\spinct_1$.
Let $\spinct_2$ denote the restriction of $\spincs_W$ to $\Spc(Y_2)$.
Then $d(Y_1,\spinct_1) = d(Y_2,\spinct_2)$, since $W$ is a rational homology cobordism.
Choose a sharp extension $\spincs_X \in \Spc(X)$ of $\spinct_2$.
It follows that $\spincs_W \cup \spincs_X \in \Spc(W \cup X)$ restricts to $\spinct_1$ and satisfies $c_1(\spincs_W \cup \spincs_X)^2 - b_2(X \cup W) = c_1(\spincs_X)^2 - b_2(X) = 4 d(Y_2,\spinct_2) = 4 d(Y_1,\spinct_1)$.
Hence $\spincs_W \cup \spincs_X$ is a sharp extension of $\spinct_1$.
Since $\spinct_1$ was arbitrary, it follows that $W \cup X$ is sharp, as desired.
\end{proof}

The next result combines \Cref{lem:hereditarysharp} and \Cref{prop:rigid}.

\begin{lem}
\label{lem:stabilization}
Suppose that $L_1$ and $L_2$ are non-split alternating links, $X_2$ is a sharp 4-manifold filling $\Sigma(L_2)$, and $W : \Sigma(L_1) \to \Sigma(L_2)$ is a quasi-ribbon cobordism.
Then the intersection pairing on $W \cup X_2$ is isometric to a stabilization of $\Lambda(L_1)$.
\end{lem}

\begin{proof}
By \Cref{lem:hereditarysharp}, the 4-manifold $W \cup X_2$ is a sharp filling of $\Sigma(L_1)$.
Consequently, its $d$-invariant is isomorphic to the $d$-invariant of $\Sigma(L_1)$, which is in turn isomorphic to the $d$-invariant of $\Lambda(L_1)$, by the example preceding \Cref{thm:GJ}.
As $\Lambda(L_1) \simeq \Flow(G_1)$, where $G_1$ is the Tait graph of a reduced alternating diagram of $L_1$, the result  follows from \Cref{prop:rigid}.
\end{proof}

\begin{proof}
[Proof of \Cref{thm:ribboncubiquity}]
Let $X_2$ be the Gordon-Litherland manifold filling $\Sigma(L_2)$.
Since $b_1(W) = b_2(W) = 0$, the embedding $X_2 \into W \cup X_2$ induces a finite index embedding of intersection lattices.
By the construction of $X_2$ and \Cref{lem:stabilization}, we obtain a finite index embedding $\Lambda(L_2) \into \Lambda(L_1) \oplus \bZ^n$ for some $n \ge 0$.
We proceed to show that this embedding is cubiquitous.

Abbreviate $\Lambda_1 = \Lambda(L_1)$ and let $\Lambda_2$ denote the image of $\Lambda(L_2)$ under the embedding.
As above, the restriction map $\Spc(W \cup X_2) \to \Spc(\Sigma(L_1))$ surjects, setting up a one-to-one correspondence between $\Spc(\Sigma(L_1))$ and the cosets of $2 (\Lambda_1 \oplus \bZ^n)$ in $\Char(\Lambda_1 \oplus \bZ^n)$.
As in the proof of \Cref{thm:GJ}, there is a one-to-one correspondence between the image of the restriction map $\Spc(W \cup X_2) \to \Spc(\Sigma(L_2))$  and $\Char(\Lambda_1 \oplus \bZ^n) / 2 \Lambda_2$.
Thus, each coset of $2 (\Lambda_1 \oplus \bZ^n)$ in $\Char(\Lambda_1 \oplus \bZ^n)$ decomposes into the disjoint union of $[\Lambda_1 \oplus \bZ^n : \Lambda_2]$ cosets of $2 \Lambda_2$ in $\Char(\Lambda_1 \oplus \bZ^n)$.

Since $X_2$ is sharp, each $\spinct_2 \in \Spc(Y_2)$ in the image of the restriction map from $\Spc(W \cup X_2)$ admits a sharp extension to $X_2$ and hence to $W \cup X_2$.
It follows that each coset of $2 \Lambda_2$ in $\Char(\Lambda_1 \oplus \bZ^n)$ contains an element of $\Short(\Lambda_1 \oplus \bZ^n) = \Short(\Lambda_1) \oplus \Short(\bZ^n)$.

By \Cref{prop:uniqueshort}, we may select some $\chi_0 \in \Short(\Lambda_1)$ which is unique in its coset in $\cX(\Lambda_1)$.
Set $\xi_0 = (-1,\dots,-1) \in \Char(\bZ^n)$.
Since $\chi_0$ is the unique element of $\Short(\Lambda_1)$ in the coset $\chi_0 + 2 \Lambda_1 \in \cX(\Lambda_1)$, the elements of $\Short(\Lambda_1 \oplus \bZ^n)$ in the coset $(\chi_0,\xi_0) + 2 (\Lambda_1 \oplus \bZ^n) \in \cX(\Lambda_1 \oplus \bZ^n)$ are the elements of $\chi_0 + \{\pm1\}^n$.
Consider the cosets of $2 \Lambda_2$ comprising $(\chi_0,\xi_0) + 2 (\Lambda_1 \oplus \bZ^n)$.
As stated above, each one contains an element of $\Short(\Lambda_1 \oplus \bZ^n)$, hence an element of $\chi_0 + \{\pm1 \}^n$.
It follows that each coset of $\Lambda_2$ in $\Lambda_1 \oplus \bZ^n$ contains an element of the form $\frac12[(\chi_0,\xi) - (\chi_0,\xi_0)] = \frac12 (0,\xi-\xi_0)$ with $\xi \in \{\pm1\}^n$.
That is, each coset of $\Lambda_2$ in $\Lambda_1 \oplus \bZ^n$ contains an element in $(0) \oplus \{0,1\}^n$.
Thus, for each $x \in \Lambda_1 \oplus \bZ^n$, the coset $-x + \Lambda_2$ contains an element in $(0) \oplus \{0,1\}^n$, which means that $x + \{0,1\}^n$ contains a point of $\Lambda_2$.
Hence $\Lambda_2$ is a cubiquitous sublattice of $\Lambda_1 \oplus \bZ^n$.
\end{proof}


\section{Examples and a conjecture}
\label{sec:conj}

In this section we highlight  some interesting examples of alternating knots and links with normalized determinant less than one.  In particular, contrasting with both Lisca's results for 2-bridge links and with the result for $\ud=1$ in \Cref{thm:links2}, we have:

\begin{prop}
\label{prop:qhbnonslice}
There exist alternating knots and links which are not $\chi$-slice but whose double branched covers are rationally nullbordant.
\end{prop}

\Cref{prop:qhbnonslice} has been established by \cite{ammmps,sartori,simone}: letting $K_n$ be the closure of the 3-braid $(\sigma_1\sigma_2^{-1})^n$, it is shown in \cite{simone} that $\Sigma(K_n)$ bounds a rational homology ball for odd $n$.  Sartori showed in \cite{sartori} that $K_7$ is nonslice and Aceto-Meier-Miller-Miller-Park-Stipsicz generalised this in \cite{ammmps}, showing that $K_n$ is nonslice for $n\in\{7,11,17,23\}$.  We refer to $K_n$ as a \emph{wheel link}, as its black graph is a wheel graph (a cycle together with an additional vertex with an edge to each vertex in the cycle).

We give a different proof that the double branched covers of odd wheel links bound rational homology balls.  This is based on Figure \ref{fig:wheels} and the following lemma.  We refer the reader to \cite{mut} for details on Conway mutation.

\begin{figure}
\includegraphics[width=\textwidth]{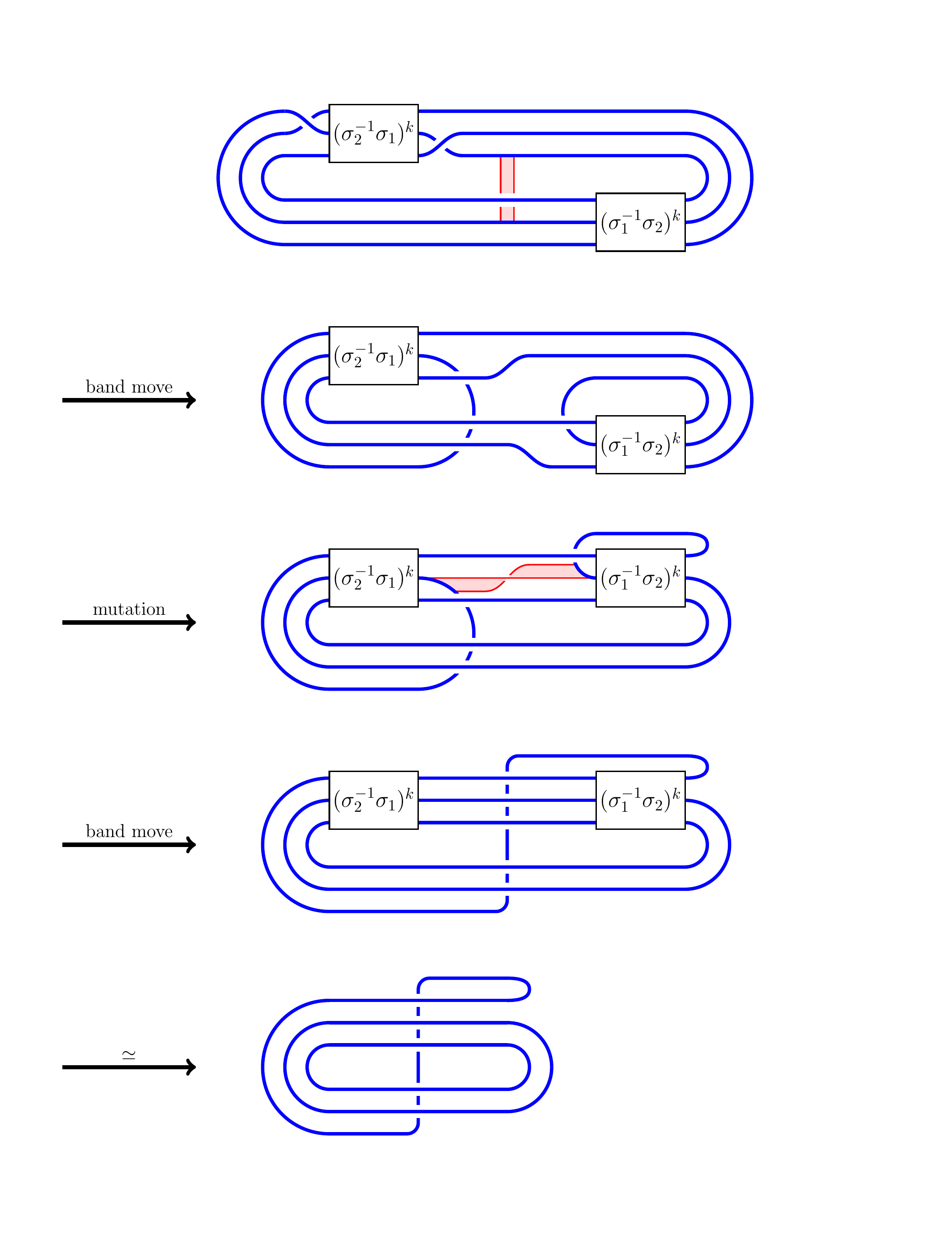}
\caption{
\label{fig:wheels}
{\bf The double branched cover of the wheel link $K_n=\overline{(\sigma_1\sigma_2^{-1})^n}$, for odd $n$, bounds a ribbon rational homology ball.} 
}
\end{figure}

\begin{lem}
\label{lem:mutball}
Let $L$ be a  link in $S^3$ with nonzero determinant, so that its double branched cover $\Sigma(L)$ is a rational homology sphere.  Suppose that there exists a finite sequence consisting of $n$ band moves and some number of Conway mutations, in any order, which converts $L$ to the $(n+1)$-component unlink.  Then $\Sigma(L)$ bounds a ribbon rational homology 4-ball.
\end{lem}
\begin{proof}  The proof is essentially the same as that given by Lisca for the case without Conway mutations \cite[Proof of Theorem 1.2]{lisca}.  Conway mutation does not change the diffeomorphism type of the double branched cover, so by taking double branched covers of the minima bounded by the unlink and of the saddle cobordisms corresponding to the band moves, we obtain a connected smooth 4-manifold $W$ bounded by $\Sigma(L)$ which has a handlebody decomposition with two 0-handles, $(n+1)$ 1-handles, and $n$ 2-handles.  Using the long exact sequence of the pair $(W,\partial W)$ we obtain $b_1(W)=b_2(W)=0$.
\end{proof}

\begin{ques} 
\label{ques:mut}
Does every ribbon rational homology ball bounded by the double branched cover $\Sigma(L)$ of an alternating link arise as in \Cref{lem:mutball}?
\end{ques}

We suspect that the answer to \Cref{ques:mut} is no, based on the following knots for which we have not succeeded in finding a sequence of band moves and mutations as in \Cref{lem:mutball}.

\begin{example}
\label{ex:12cr}
The 12-crossing knots $12a360$ and $12a1237$ are nonslice, but their double branched covers bound rational homology balls.
\end{example}

Nonsliceness of these knots is recorded in \cite{knotinfo} and may be verified using the Fox-Milnor obstruction and Levine-Tristram signatures. The rational balls for \Cref{ex:12cr} are exhibited by the Kirby diagrams shown in Figure \ref{fig:12aQHB}.  In each case the diagram yields an embedding of the Gordon-Litherland manifold $X=\Sigma(B)$, the double branched cover of the 4-ball branched along the black surface of an alternating diagram, into a connected sum of $b_2(X)$ copies of $\C\P^2$.  The complement of $X$ in
$\#_{b_2(X)}\C\P^2$ is then the required rational homology ball.  In each case this complement has a handle decomposition relative to its boundary with handles of index at least 2; turning this upside-down we see that we have a ribbon rational ball.  The method for drawing a Kirby diagram of the Gordon-Litherland manifold is described in \cite[Lecture 2]{wb}, and is based on a description given in \cite{HFdbc}.

\begin{figure}
\includegraphics[scale=0.4]{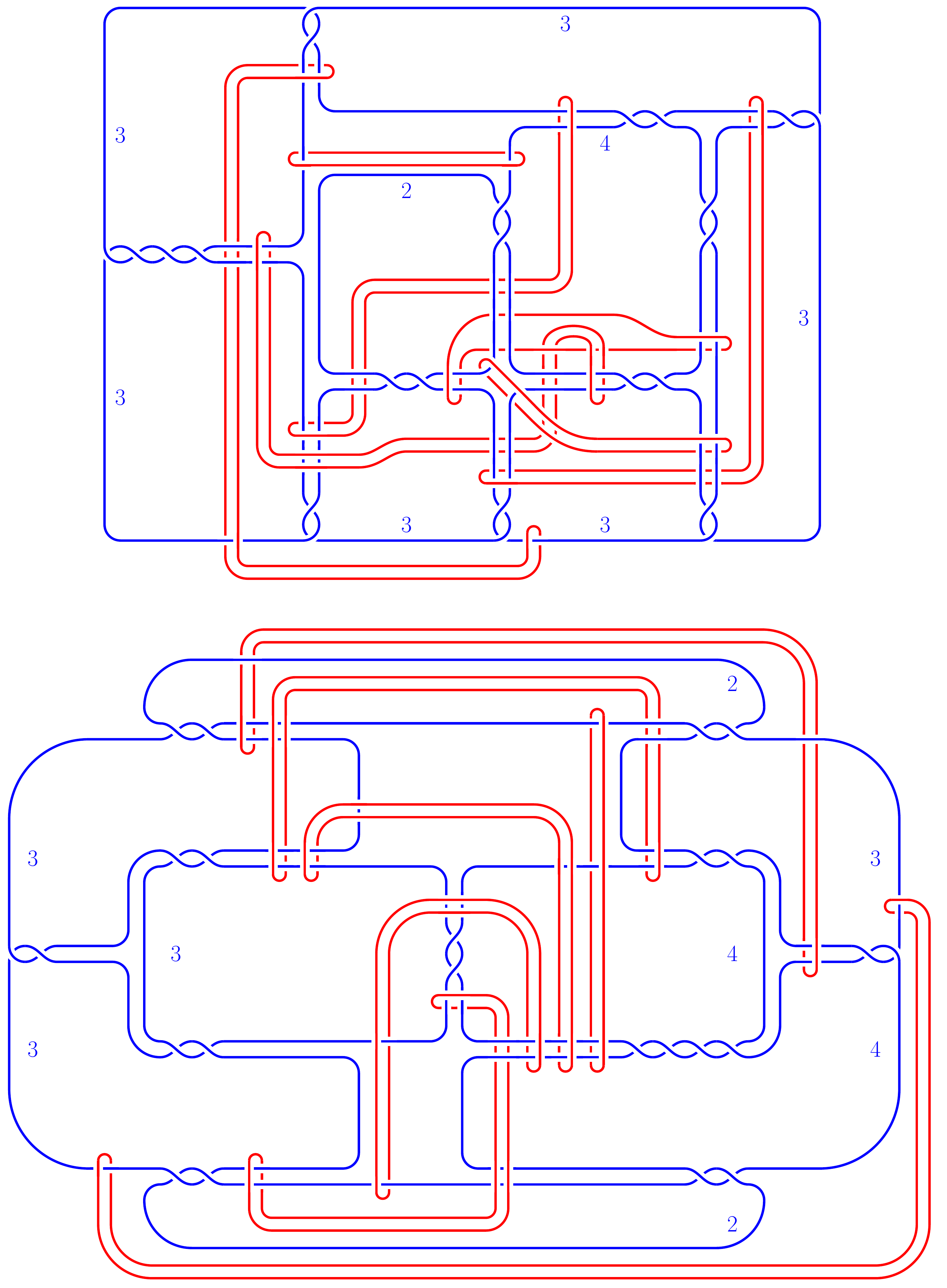}
\caption{
\label{fig:12aQHB}
{\bf Embeddings of manifolds $X\hookrightarrow \#_7\C\P^2$ exhibiting ribbon rational homology balls bounded by $\Sigma(12aN)$, for $N\in\{360,1237\}$.} In each case the manifold $X$ is given by the blue 2-handles and a 3-handle, and is the double cover of the 4-ball branched over the black surface of $12aN$.  Each red 2-handle attaching circle is 1-framed.  Blowing these down results in an 8-component 0-framed unlink, from which it follows that the manifold given by attaching the red 2-handles, seven further 3-handles, and a 4-handle to $X$ is $\#_7\C\P^2$.
}
\end{figure}

The diagrams in Figure \ref{fig:12aQHB} were found by a thus-far ad hoc procedure, which has also been applied successfully to the wheel links mentioned above, and which is based on the notion of alternating links and lattice embeddings being \emph{expanded from the unknot}.  This in turn is based on a generalisation of the expansions used by Lisca in \cite{lisca}, which we now describe.

We first briefly consider the effect of edge deletions and contractions on the flow lattice of a graph.  Contracting a single edge leaves the rank of the flow lattice unchanged, and the pairing of two cycles is unchanged unless both contain the contracted edge, in which case it is changed by $\pm1$.  Deletion of a non-cut edge reduces the rank of the flow lattice by 1; the new flow lattice is a sublattice of the original, with quotient generated by any cycle containing the deleted edge.  Thus if $G'$ is obtained from $G$ by contracting some edges and deleting some non-cut edges of $G$, then there is an induced group monomorphism 
$\Phi:\flow(G')\to\flow(G)$.

If $D$ is an alternating link diagram and $G$ is its black Tait graph, then smoothing of crossings in $D$ corresponds to contraction and deletion of edges in $G$: the edge corresponding to a crossing is contracted if the crossing is smoothed so as to join two black regions in the diagram, and it is deleted if two white regions are joined.

Suppose now that $L_i$ are nonsplit alternating links, with alternating diagrams $D_i$ for $i=1,2$,  and suppose that there exist embeddings $\phi_i:\Lambda(L_i)\hookrightarrow\zz^{n_i}$, where $n_i$ is the rank of the black lattice $\Lambda(L_i)$.
We say that the pair $(D_2,\phi_2)$ is obtained by an  expansion from $(D_1,\phi_1)$ if
\begin{enumerate}[(i)]
\item $D_1$ is obtained from $D_2$ by smoothing two crossings, so that one smoothing joins two white regions and the other joins two black regions and 
\item the following diagram of group homomorphisms commutes, where $\pi$ is projection orthogonal to a unit vector:
\begin{center}
\begin{tikzcd}
\Lambda_1 \arrow[hookrightarrow]{r}{\phi_1} \arrow{d}{\Phi} 
  & \zz^{n_1}\\
\Lambda_2 \arrow[hookrightarrow]{r}{\phi_2}
  & \zz^{n_2=n_1+1}\arrow{u}{\pi} .
\end{tikzcd}
\end{center}
\end{enumerate}
We say that $(D',\phi')$ is \emph{expanded from} $(D,\phi)$ if there is a sequence 
$$(D',\phi')=(D_k,\phi_k),(D_{k-1},\phi_{k-1}),\dots,(D_1,\phi_1)=(D,\phi)$$
with  $(D_{j+1},\phi_{j+1})$ obtained by an expansion from $(D_j,\phi_j)$ for each $j=1,\dots,k-1$.
Finally we say that an alternating link $L$ is \emph{expanded from the unknot} if there is a finite-index embedding $\phi:\Lambda(L)\hookrightarrow\zz^n$, an alternating diagram $D$ of $L$, and an alternating diagram $D_0$ of the unknot, such that
$(D,\phi)$ is expanded from $(D_0,\Id)$, noting that the black lattice of an alternating diagram of the unknot is Euclidean.

We have the following generalization of \Cref{conj:main}.

\begin{conj}
\label{conj:exp}
Let $L$ be a nonsplit alternating link with alternating diagram $D$.  The following are equivalent:
\begin{enumerate}[(1)]
\item $\Sigma(L)$ is rationally nullbordant;
\item $\Lambda(L)$ is cubiquitous;
\item $L$ is expanded from the unknot.
\end{enumerate}
\end{conj}

The implication (1)$\implies$(3) of this conjecture is true for 2-bridge links \cite{lisca}, and also for alternating links with $\ud=1$.  The latter follows from the proof of \Cref{thm:links2}.  The equivalence of (1) and (3) has  been verified  for odd wheel links, and for prime alternating knots with 12 or fewer crossings, including $12a360$ and $12a1237$.  As before, we have (1)$\implies$(2) by the work of Greene and Jabuka, \Cref{thm:GJ}. The implication (2)$\implies$(1) is also known to be true for 2-bridge links, though this is not written down in full \cite{noteond,lisca}.  With care one can sometimes recursively build up Kirby diagrams as in Figure \ref{fig:12aQHB}   by working one expansion at a time.  One might hope that this approach could lead to a proof of (3)$\implies$(1) in \Cref{conj:exp}.


\clearpage

\bibliographystyle{amsplain}
\bibliography{tiling}

\end{document}